%% file: main.tex
\documentclass[onefignum,onetabnum]{parts/siamonline220329}

\input{parts/preamble}

\ifpdf
\hypersetup{
  pdftitle={When big data actually are low-rank, or entrywise approximation of certain function-generated matrices},
  pdfauthor={Stanislav Budzinskiy}
}
\fi

\begin{document}

\maketitle

\begin{abstract}
The article concerns low-rank approximation of matrices generated by sampling a smooth function of two $m$-dimensional variables. We identify several misconceptions surrounding a claim that, for a specific class of analytic functions, such $n \times n$ matrices admit accurate entrywise approximation of rank that is independent of $m$ and grows as $\log(n)$---colloquially known as ``big-data matrices are approximately low-rank''. We provide a theoretical explanation of the numerical results presented in support of this claim, describing three narrower classes of functions for which function-generated matrices can be approximated within an entrywise error of order $\varepsilon$ with rank $\mathcal{O}(\log(n) \varepsilon^{-2} \log(\varepsilon^{-1}))$ that is independent of the dimension $m$: (i)~functions of the inner product of~the~two variables, (ii)~functions of the Euclidean distance between the variables, and (iii)~shift-invariant positive-definite kernels. We extend our argument to tensor-train approximation of tensors generated with functions of the ``higher-order inner product'' of their multiple variables. We discuss our results in the context of low-rank approximation of (a)~growing datasets and (b)~attention in transformer neural networks.
\end{abstract}



\input{parts/introduction}
\input{parts/general_functions}
\input{parts/dotp}
\input{parts/dist}
\input{parts/kernels}
\input{parts/tensors}
\input{parts/numerical}
\input{parts/conclusion}
\input{parts/acknowledgements}
\appendix
\input{parts/appendix}

\bibliographystyle{parts/siamplain}
\bibliography{}
\end{document}

%% file: parts/preamble.tex
\usepackage{amsfonts, amssymb, amsopn}
\usepackage{bm}
\usepackage{graphicx}
\usepackage{csquotes}
\usepackage{stmaryrd}
\usepackage{epstopdf}
\usepackage{algorithmic}
\ifpdf
  \DeclareGraphicsExtensions{.eps,.pdf,.png,.jpg}
\else
  \DeclareGraphicsExtensions{.eps}
\fi

\renewenvironment{equation*}{\[}{\]\ignorespacesafterend}

\newsiamremark{remark}{Remark}
\newsiamremark{example}{Example}
\newsiamremark{assumption}{Assumption}
\crefname{assumption}{Assumption}{Assumptions}

\headers{When big data actually are low-rank}{Stanislav Budzinskiy}

\title{When big data actually are low-rank, or entrywise approximation of certain function-generated matrices\thanks{This research was funded in part by the Disruptive Innovation – Early Career Seed Money funding program of the Austrian Academy of Sciences (\"OAW) and the Austrian Science Fund (FWF). This research was funded in part by the Austrian Science Fund (FWF) \href{https://doi.org/10.55776/F65}{10.55776/F65}. For open access purposes, the author has applied a CC BY public copyright license to any author-accepted manuscript version arising from this submission.}}

\author{
Stanislav Budzinskiy\thanks{Faculty of Mathematics, University of Vienna, Kolingasse 14-16, 1090 Vienna, Austria (\email{stanislav.budzinskiy\allowbreak@univie.ac.at}).}
}

\DeclareMathOperator{\diag}{diag}

\newcommand{\Real}{\mathbb{R}}
\newcommand{\N}{\mathbb{N}}

\newcommand{\hap}{\mathbin{\ast}}
\newcommand{\khrap}{\mathbin{\odot}}
\newcommand{\krp}{\mathbin{\otimes}}

\newcommand{\Ell}{\mathrm{L}}

\newcommand{\Point}[1]{\bm{\mathrm{#1}}}
\newcommand{\Matrix}[1]{\bm{\mathrm{#1}}}
\newcommand{\Tensor}[1]{\bm{\mathcal{#1}}}
\newcommand{\trans}{\intercal}

\newcommand{\ball}[1][]{#1{{\mathbb{B}}}}

\newcommand{\set}[3][]{#1\{ #2 : #3 #1\}}

\newcommand{\Norm}[3][]{#1\| #2 #1\|_{#3}}

\newcommand{\rank}[2][]{\mathrm{rank} #1( #2 #1)}
\newcommand{\ttrank}[2][]{\mathrm{rank_{TT}} #1( #2 #1)}
\newcommand{\epsrank}[2][]{\mathrm{rank_{\varepsilon}} #1( #2 #1)}


%% file: parts/introduction.tex
\section{Introduction}

Datasets are often arranged as matrices $\Matrix{F} \in \Real^{n_1 \times n_2}$ and obtained by (or at least modeled as the result of) sampling a bi-variate function $f : \Omega_1 \times \Omega_2 \to \Real$ so that 
\begin{equation*}
    \Matrix{F}(i,j) = f(\Point{x}_i, \Point{y}_j), \quad \{ \Point{x}_i \}_{i = 1}^{n_1} \subset \Omega_1, \quad \{ \Point{y}_j \}_{j = 1}^{n_2} \subset \Omega_2.
\end{equation*}
We shall say that such $\Matrix{F}$ is a function-generated matrix and $f$ is its generating function. Given the sheer volume of modern datasets, it is necessary to use low-parametric representation formats for matrices, the most prominent being approximate low-rank decompositions:
\begin{equation*}
    \Matrix{F} \approx \Matrix{A} \Matrix{B}^\trans, \quad \Matrix{A} \in \Real^{n_1 \times r}, \quad \Matrix{B} \in \Real^{n_2 \times r}, \quad r \ll \min\{ n_1, n_2 \}.
\end{equation*}
When a matrix admits accurate low-rank approximation, its storage and processing become more efficient, noise can be filtered out from its entries, and missing entries can be imputed.

How does the accuracy of low-rank approximation of a matrix depend on the properties of its generating function? Since we treat the matrix as a dataset, we consider the error in the maximum norm $\Norm{\Matrix{F}}{\max} = \max_{i \in [n_1],j \in [n_2]} |\Matrix{F}(i,j)|$. Small approximation error in the maximum norm guarantees that each entry is perturbed only slightly. The maximum norm goes hand in hand with the $\Ell_\infty$ norm of bounded functions $\Norm{f}{\Ell_\infty(\Omega_1 \times \Omega_2)} = \sup_{\Point{x} \in \Omega_1,\Point{y} \in \Omega_2} |f(\Point{x},\Point{y})|$. Indeed, if $\Matrix{F}$ is generated with $f$ and $\Matrix{G} \in \Real^{n_1 \times n_2}$ is generated with $g : \Omega_1 \times \Omega_2 \to \Real$, but with the same sampling points, then $\Norm{\Matrix{F} - \Matrix{G}}{\max} \leq \Norm{f - g}{\Ell_\infty(\Omega_1 \times \Omega_2)}$.

\begin{example}
\label{example:1d}
Let $\ball_1 = (-1,1)$ and consider a bounded function $f: \ball_1^2 \to \Real$. Assume that $y \mapsto f(x, y) \in C^{\infty}(\ball_1)$ for each $x \in \ball_1$ and its derivatives grow as
\begin{equation*}
    \Norm{\partial_y^\gamma f(x, \cdot)}{\Ell_\infty(\ball_1)} \leq C M^{\gamma} \Norm{f}{\Ell_\infty(\ball_1^2)}, \quad \gamma \in \N_0,
\end{equation*}
with constants $C \geq 1$ and $M > 0$. Then the Taylor series of $y \mapsto f(x, y)$ at $y = 0$ satisfies
\begin{equation*}
    f(x,y) = \sum\nolimits_{\gamma = 0}^{r-1} \partial_y^\gamma f(x, 0) \frac{y^\gamma}{\gamma!} + E_r(x,y), \quad \Norm{E_r(x,\cdot)}{\Ell_\infty(\ball_1)} \leq C \frac{M^r}{r!} \Norm{f}{\Ell_\infty(\ball_1^2)}.
\end{equation*}
Stirling's formula $r! > \sqrt{2 \pi r} (r/e)^r$ gives $\Norm{E_r}{\Ell_\infty(\ball_1^2)} < C e^{-r} \Norm{f}{\Ell_\infty(\ball_1^2)}$ when $r \geq e^2 M$.
\end{example}

The truncated Taylor series $f_r(x,y) = \sum_{\gamma = 0}^{r-1} \partial_y^\gamma f(x, 0) \frac{y^\gamma}{\gamma!}$ is a sum of $r$ terms with separated variables. When sampled at points $\{ x_i \}_{i = 1}^{n_1} \subset \ball_1$ and $\{ y_j \}_{j = 1}^{n_2} \subset \ball_1$, it produces a matrix 
\begin{equation*}
    \Matrix{G} = 
    \begin{bmatrix}
        \partial_y^0 f(x_1, 0) & \dots & \partial_y^{r-1} f(x_1, 0) \\
        \vdots & \ddots & \vdots \\
        \partial_y^0 f(x_{n_1}, 0) & \dots & \partial_y^{r-1} f(x_{n_1}, 0) \\
    \end{bmatrix}
    \begin{bmatrix}
        \frac{1}{0!} & & \\
        & \ddots & \\
        & & \frac{1}{(r-1)!}
    \end{bmatrix}
    \begin{bmatrix}
        y_1^0 & \dots & y_1^{r-1} \\
        \vdots & \ddots & \vdots \\
        y_{n_2}^0 & \dots & y_{n_2}^{r-1} \\
    \end{bmatrix}^\trans
\end{equation*}
of $\rank{\Matrix{G}} \leq r$ such that $\Norm{\Matrix{F} - \Matrix{G}}{\max} \leq C e^{-r} \Norm{f}{\Ell_\infty(\ball_1^2)}$. The smoothness of $f$ guarantees that the error of order $\varepsilon$ can be achieved with a matrix of rank at most $\max\{\log(C/\varepsilon), e^2 M\}$.

In many situations, the two variables of $f$ are multi-dimensional. Take, for example, the integral kernel of the Newtonian potential and kernel functions used in machine learning. How do the dimensions of the variables affect the approximability properties of the matrix? Below, we use the standard multi-index notation (see \cref{appendix:multindex}).

\begin{example}
\label{example:multi_dimensions}
Let $\ball_m = \set{\Point{x} \in \Real^m}{\Norm{\Point{x}}{2} < 1}$ and 
consider a bounded function $f: \ball_m^2 \to \Real$. Assume that $\Point{y} \mapsto f(\Point{x}, \Point{y}) \in C^{\infty}(\ball_m)$ for each $\Point{x} \in \ball_m$ and its partial derivatives grow as 
\begin{equation*}
    \Norm{\partial_{\Point{y}}^{\bm{\gamma}} f(\Point{x}, \cdot)}{\Ell_\infty(\ball_m)} \leq C M^{|\bm{\gamma}|} \Norm{f}{\Ell_\infty(\ball_m^2)}, \quad \bm{\gamma} \in \N_0^m,
\end{equation*}
with constants $C \geq 1$ and $M > 0$. Then the Taylor series of $\Point{y} \mapsto f(\Point{x}, \Point{y})$ at $\Point{y} = \Point{0}$ satisfies
\begin{equation*}
    f(\Point{x}, \Point{y}) = \sum\nolimits_{|\bm{\gamma}| < \rho} \partial_{\Point{y}}^{\bm{\gamma}} f(\Point{x}, \Point{0}) \frac{\Point{y}^{\bm{\gamma}}}{\bm{\gamma}!} + E_{\rho}(\Point{x}, \Point{y}), \quad \Norm{E_{\rho}(\Point{x},\cdot)}{\Ell_\infty(\ball_m)} \leq C \frac{m^{\rho} M^{\rho}}{{\rho}!} \Norm{f}{\Ell_\infty(\ball_m^2)}.
\end{equation*}
Stirling's formula gives $\Norm{E_{\rho}}{\Ell_\infty(\ball_m^2)} < C e^{-\rho} \Norm{f}{\Ell_\infty(\ball_m^2)}$ when $\rho \geq e^2 m M$. The truncation parameter $\rho$ and the number of terms $r$ in the decomposition satisfy $r = \genfrac{(}{)}{0pt}{1}{m + \rho - 1}{m}$ \cite[\S B.1]{adcock2022sparse}.
\end{example}

This purely analytical approach leads to poor approximation guarantees for matrices generated with functions of high-dimensional variables: it provides a matrix $\Matrix{G}$ of rank\footnote{We use $\genfrac{(}{)}{0pt}{1}{n}{k} \leq (e n / k)^k$ and $1 + x < e^x$ in the inequalities.}
\begin{equation}
\label{eq:analytical_bound}
    \rank{\Matrix{G}} \leq \genfrac{(}{)}{0pt}{1}{m + \rho_{\ast} - 1}{m} \leq e^m (1 + \tfrac{\rho_{\ast} - 1}{m})^m < e^{m + \rho_{\ast} - 1}, \quad \rho_{\ast} =  \lceil \max\{ \log(C / \varepsilon), e^2 m M \} \rceil,
\end{equation}
that achieves the entrywise approximation error of order $\varepsilon$. In this article, we investigate how \emph{combinations of analytical and algebraic techniques} can be used to derive better rank bounds for the entrywise approximation of function-generated matrices.

\subsection{Random embeddings}
The algebraic technique that we are going to use is based on random embeddings. Let $\Matrix{F} = \Matrix{A} \Matrix{B}^\trans$, where $\Matrix{A}$ and $\Matrix{B}$ can have an astronomical number of columns relative to the size of $\Matrix{F}$. Taking a random matrix $\Matrix{R}$ with independent Gaussian entries, we can guarantee that $\Matrix{F} \approx (\Matrix{A}\Matrix{R}) (\Matrix{B}\Matrix{R})^\trans$ entrywise with constant positive probability.

\begin{theorem}[{\cite[Lemma~5]{srebro2005rank}}]
\label{theorem:compression}
Let $\varepsilon \in (0,1)$ and $n_1, n_2 \in \N$. Consider
\begin{equation}
\label{eq:compression_rank}
    r = \left\lceil 9 \log\left(3 n_1 n_2\right) / \varepsilon^2 \right\rceil \in \N.
\end{equation}
For every $m \in \N$ and every pair of $\Matrix{A} \in \Real^{n_1 \times m}$ and $\Matrix{B} \in \Real^{n_2 \times m}$, there exists $\Matrix{G} \in \Real^{n_1 \times n_2}$ of $\rank{\Matrix{G}} \leq r$ such that
\begin{equation*}
    \Norm{\Matrix{A} \Matrix{B}^\trans - \Matrix{G}}{\max} \leq \varepsilon \Norm{\Matrix{A}}{2,\infty} \Norm{\Matrix{B}}{2,\infty}.
\end{equation*}
\end{theorem}
\cref{theorem:compression} was proved with the Johnson--Lindenstrauss lemma and was rediscovered in the specific cases of symmetric positive-semidefinite decompositions \cite{alon2013approximate} and the singular value decomposition \cite{udell2019big}. A different proof based on the Hanson--Wright inequality was proposed in \cite{budzinskiy2024distance} and further extended to tensor-train (TT, \cite{oseledets2009breaking,oseledets2011tensor}) approximation of tensors in \cite{budzinskiy2025entrywise}.

The error bound in \cref{theorem:compression} is closely related with the factorization norm\footnote{It is also known as the max-norm \cite{srebro2005rank}, which should not be confused with $\Norm{\cdot}{\max}$ that we use.} \cite{linial2007lower}
\begin{equation*}
    \gamma_2(\Matrix{F}) = \inf \set{\Norm{\Matrix{A}}{2,\infty} \Norm{\Matrix{B}}{2,\infty}}{\Matrix{F} = \Matrix{A} \Matrix{B}^\trans}.
\end{equation*}
This norm is used in matrix completion \cite{srebro2004maximum, rennie2005fast, lee2010practical, cai2016matrix, foucart2020weighted} and communication complexity \cite{linial2007complexity, linial2007lower, matouvsek2020factorization} and has been extended to tensors \cite{ghadermarzy2019near, harris2021deterministic, budzinskiy2025entrywise, cao20241} with applications to tensor~completion.

\subsection{Contributions}
\label{subsec:contrib}
Random embeddings were used in \cite{udell2019big} to derive a new rank bound for matrices generated with functions from \cref{example:multi_dimensions}. In \Cref{sec:general_functions}, we carefully examine the results of \cite{udell2019big} and the informal interpretations presented therein. By analyzing the rank bound of \cite{udell2019big}, we identify a number of inaccurate claims and clarify them.

We proceed by studying three specific, narrower classes of function-generated matrices for which we can derive low-rank approximation guarantees with rank bounds of the form
\begin{equation}
\label{eq:rank_bound}
    \mathcal{O}\left( \log(n) \varepsilon^{-2} \log\left(\varepsilon^{-1}\right) \right),
\end{equation}
where the hidden constant is independent of $m$. The corresponding generating functions are
\begin{itemize}
    \item $f(\Point{x}, \Point{y}) = h(\Point{x}^\trans \Point{y})$ in \cref{theorem:1d_function_dotp_gen},
    \item $f(\Point{x}, \Point{y}) = h(\Norm{\Point{x} - \Point{y}}{2}^2)$ in \cref{theorem:1d_function_dist_gen,corollary:1d_function_dist_gen_shifted}
\end{itemize}
with $h : \Real \to \Real$ satisfying \Cref{assumption:1d_analytic}. We derive a tighter rank bound $\mathcal{O}\left( \log(n) \varepsilon^{-2} \right)$ for
\begin{itemize}
    \item positive-definite kernels $f(\Point{x}, \Point{y}) = \kappa(\Point{x} - \Point{y})$ in \cref{theorem:kernel_function}.
\end{itemize}
We then consider function-generated tensors and prove in \cref{theorem:1d_function_tensor} that if
\begin{equation*}
    f(\Point{x}^{(1)}, \ldots, \Point{x}^{(d)}) = h(\langle \Point{x}^{(1)}, \ldots, \Point{x}^{(d)} \rangle), \quad \Point{x}^{(i)} \in \Real^m,
\end{equation*}
depends only on the ``higher-order inner product'' \cref{eq:hoip} of the variables then there exists a TT approximation that achieves entrywise error of order $\varepsilon$ with TT rank bounded componentwise as in \cref{eq:rank_bound}, where the hidden constant depends on $d$ but is independent of $m$. 

These low-rank approximation guarantees are primarily of theoretical interest. We compare them with the results of numerical experiments for moderate $n$ in \Cref{sec:numerical}.

\subsection{Contributions: growing datasets}
\label{subsec:datasets}
Consider a family of matrices $\{ \Matrix{F}_n \}_{n \in \N}$ of size $\Matrix{F}_n \in \Real^{n \times n}$ generated with the same function from one of the classes described in \Cref{subsec:contrib}. These matrices serve as a model of a growing dataset that is continuously augmented with new samples. Our results show that the sequence of $\varepsilon$-ranks, 
\begin{equation}
\label{eq:epsrank}
    \epsrank{\Matrix{F}_n} = \min\set{\rank{\Matrix{G}_n}}{\Norm{\Matrix{F}_n - \Matrix{G}_n}{\max} < \varepsilon},
\end{equation}
can be upper bounded by a sequence $\{ r_n \}_{n \in \N}$ that evolves in three stages as $n \to \infty$ (\cref{fig:dataset}). 

\begin{figure}[h]
\centering
	\includegraphics[width=0.65\textwidth]{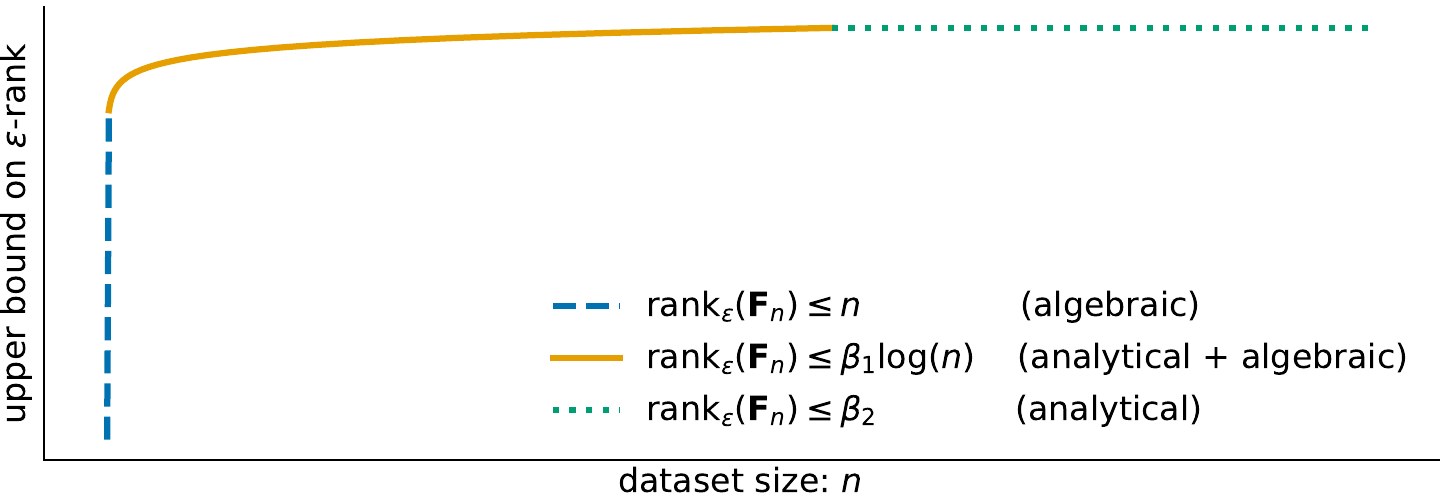}
\caption{Three-stage evolution of the $\varepsilon$-rank of a dataset as new samples are accumulated (schematic plot).}
\label{fig:dataset}
\end{figure}

For small $n$, the best (uniformly with respect to the collected samples) bound is algebraic: we cannot guarantee that $\Matrix{F}_n$ can be approximated with a low-rank matrix, so $r_n = n$.

The second, intermediate stage begins when our rank bound \cref{eq:rank_bound}, obtained with a combination of analytical and algebraic techniques, becomes tighter than $n$. From this point on, the bound on the $\varepsilon$-ranks grows slowly as $r_n = \beta_1 \log(n)$ with a constant $\beta_1 > 0$ that is independent of the dimension $m$. Notably, the critical size $n$ at which the transition takes place depends only on the properties of the function $h : \Real \to \Real$ (see \Cref{subsec:contrib}) and not on the dimension $m$ of the variables of $f(\Point{x}, \Point{y})$.

Finally, for sufficiently large $n$, the logarithmic upper bound starts to exceed\footnote{Note that $m$ needs to be sufficiently large in order for the logarithmic stage to be present.} the analytical rank bound \cref{eq:analytical_bound}, $\beta_2 \in \N$, that depends on $m$ but not on $n$. As a result, $r_n = \beta_2$ for all $n$ as $n \to \infty$.  For the specific classes of generating functions described in \Cref{subsec:contrib}, the analytical rank bound \cref{eq:analytical_bound} can be improved to a polynomial in $m$ (see \Cref{sec:dist,sec:kernels}). 

It will likely be problematic to use these observations in practice due to the high computational cost of entrywise low-rank approximation (\Cref{sec:numerical}) and the fact that $r_n$ is low only relative to $n$. Indeed, consider the simplest example $f(\Point{x}, \Point{y}) = \Point{x}^\trans \Point{y}$ with sufficiently large $m$, which we can analyze directly with \cref{theorem:compression}. Let $\varepsilon = 0.1$, then the second stage begins at $n = 18964$ and we have $r_{10^5} = 21713$, $r_{10^7} = 30002$, and $r_{10^9} = 38291$. The growth is slow, but in absolute terms the ranks are large. Nonetheless, \cref{fig:varm} gives us a general idea of how the approximability of a dataset changes as new samples are accumulated. Numerical experiments in \Cref{sec:numerical} with moderate $n$, to which our reasoning does not formally apply, still show that the quality of approximation changes slowly.

The same three stages were claimed in \cite{udell2019big} to exist for the generating functions from \cref{example:multi_dimensions}, yet no proof was provided. We show in \Cref{sec:general_functions} that the logarithmic stage can exist only in very restrictive settings. In addition, if the transition from linear to logarithmic growth does happen, the critical size $n$ at which it takes place depends strongly on $m$.

\subsection{Contributions: fast attention}
\label{subsec:attention}
The transformer neural-network architecture \cite{vaswani2017attention} is widely used in natural language processing and computer vision. To process sequences of input tokens, transformers rely on the \emph{attention} mechanism. Consider $n \in \N$ tokens. With each of them, we associate three vectors of the same dimension $m \in \N$ of latent representation: a query $\Matrix{q}_i \in \Real^m$, a key $\Matrix{k}_i \in \Real^m$, and a value $\Matrix{v}_i \in \Real^m$. These vectors are stacked together as matrices $\Matrix{Q}, \Matrix{K}, \Matrix{V} \in \Real^{n \times m}$ in the definition of the dot-product attention:
\begin{equation*}
    \mathrm{Att} = \Matrix{D}^{-1} \Matrix{A} \Matrix{V}, \quad \Matrix{A} = \exp\big(\Matrix{Q} \Matrix{K}^\trans / \sqrt{m}\big), \quad \Matrix{D} = \diag\left(\sum\nolimits_j\Matrix{A}(1,j), \ldots, \sum\nolimits_j\Matrix{A}(n,j)\right),
\end{equation*}
where the exponential is applied elementwise. The straightforward evaluation of attention requires $\mathcal{O}(n^2 m)$ operations. As a result, long sequences of tokens produce a heavy computational burden for transformers.

Low-rank methods have been proposed to approximate attention with computational complexity that is subquadratic in $n$ \cite{wang2020linformer, choromanski2020rethinking, han2023hyperattention, alman2024fast, sarlos2024hardness}. In particular, \cite{choromanski2020rethinking, sarlos2024hardness} address the problem of approximating the matrix $\Matrix{A}$. The algorithmic hardness of obtaining quasi-optimal low-rank approximations of $\Matrix{A}$ in the Frobenius norm is the topic of \cite{sarlos2024hardness}. It was shown in \cite{choromanski2020rethinking} that $\Matrix{A}$ can be approximated within an entrywise error $\varepsilon$ with rank $\mathcal{O}(m \varepsilon^{-2} \log(m / \varepsilon))$ based on orthogonal random features---our \cref{theorem:1d_function_dotp_gen} guarantees a rank bound $\mathcal{O}(\log(n) \varepsilon^{-2} \log(1/\varepsilon))$, which is an improvement when the latent dimension $m$ is greater than $\log(n)$.

In higher-order tensor attention \cite{alman2023capture, sanford2024representational}, multiple keys and values are associated with every token, giving rise to key matrices $\Matrix{K}_1, \ldots, \Matrix{K}_d \in \Real^{n \times m}$ and value matrices $\Matrix{V}_1, \ldots, \Matrix{V}_d \in \Real^{n \times m}$. The higher-order attention is then defined via the Khatri--Rao product (see \cref{eq:khrap} below):
\begin{equation*}
    \mathrm{Att}_d = \Matrix{D}^{-1} \Matrix{A} (\Matrix{V}_1 \khrap \cdots \khrap \Matrix{V}_d), \quad \Matrix{A} = \exp\big(\Matrix{Q} (\Matrix{K}_1 \khrap \cdots \khrap \Matrix{K}_d)^\trans / \sqrt{m}\big).
\end{equation*}
We can recognize $\Matrix{Q} (\Matrix{K}_1 \khrap \cdots \khrap \Matrix{K}_d)^\trans$ as an unfolding matrix of an order-$(d+1)$ tensor of size $n \times \cdots \times n$ given in the canonical polyadic (CP) format \cite{kolda2009tensor,ballard2025tensor} with factors $\Matrix{Q}$, $\Matrix{K}_d$, \ldots, $\Matrix{K}_1$. By \cref{theorem:1d_function_tensor}, the matrix $\Matrix{A}$, when seen as a tensor of order $d+1$, can be approximated with TT rank independent of the latent dimension $m$. TT approximation of higher-order attention is a novel idea, which could stimulate the development of new algorithms for transformers.

\subsection{Notation}
\label{subsec:notation}
We denote the $\ell_p$ norm by $\Norm{\Point{x}}{p}$ and the norm induced by a symmetric positive-definite matrix by $\Norm{\Point{x}}{\Matrix{W}} = \sqrt{\Point{x}^\trans \Matrix{W} \Point{x}}$. We use $\ball_m = \set{\Point{x} \in \Real^m}{\Norm{\Point{x}}{2} < 1}$ and $\overline{\ball}_{m,R} = \set{\Point{x} \in \Real^m}{\Norm{\Point{x}}{2} \leq R}$. We denote the smallest and largest entries of a matrix by $\min\{\Matrix{A}\}$ and $\max\{\Matrix{A}\}$, its $\ell_p \to \ell_q$ operator norm by $\Norm{\Matrix{A}}{p,q} = \sup\set{\Norm{\Matrix{Ax}}{q}}{\Norm{\Matrix{x}}{p} = 1}$ and $\Norm{\Matrix{A}}{p} = \Norm{\Matrix{A}}{p,p}$. Note that $\Norm{\Matrix{A}}{\max} = \max\{ |\min\{\Matrix{A}\}|, |\max\{\Matrix{A}\}| \}$ and $\Norm{\Matrix{A}}{2,\infty} = \max_{i} \Norm{\Matrix{A}(i,:)}{2}$. We denote the Hadamard (entrywise) product of matrices by $\hap$ and use $\Matrix{A}^{\hap s}$ for $\Matrix{A} \hap \cdots \hap \Matrix{A}$ multiplied $s$ times. We write the Kronecker product as $\krp$ and the Khatri--Rao product as $\khrap$, where
\begin{equation}
\label{eq:khrap}
    \Matrix{A} \khrap \Matrix{B} = \begin{bmatrix}
        \Matrix{A}(:,1) \krp \Matrix{B}(:,1) & \cdots & \Matrix{A}(:,n) \krp \Matrix{B}(:,n)
    \end{bmatrix}.
\end{equation}
We use $[n] = \{ 1, \ldots, n \}$ for $n \in \N$. For a tuple of numbers, $(r_1, \ldots, r_{n}) \preccurlyeq r$ means that $r_s \leq r$ for every $s \in [n]$.

%% file: parts/general_functions.tex
\section{Functions from Example~\ref{example:multi_dimensions}}
\label{sec:general_functions}
An alternative to the rank bound \cref{eq:analytical_bound} was derived in \cite{udell2019big}. The authors study so-called \emph{latent variable models} (LVMs). In our terminology, an LVM is the set of all matrices generated with a continuous function $f : \Omega_1 \times \Omega_2 \to \Real$. An LVM is said to be \emph{nice} if $\Omega_1, \Omega_2 \subset \Real^m$ are contained in a closed ball $\overline{\ball}_{m,R}$ and the function $f$ satisfies the assumptions of \Cref{example:multi_dimensions}---the parameters of a nice LVM are $(m, R, C, M)$, where \textquote{$m$ is allowed to be extremely large} \cite[Definition~4.1]{udell2019big}.\footnote{In all quotations from \cite{udell2019big}, we use the notation of our article instead of the equivalent original notation.} The main result of \cite{udell2019big} is a rank bound for nice LVMs, which we formulate in the specific case of $R = 1$.

\begin{theorem}[{\cite[Theorem~4.2]{udell2019big}}]
\label{theorem:ut}
Let $f : \ball_m^2 \to \Real$ satisfy the assumptions of \cref{example:multi_dimensions}. For all $\varepsilon \in (0,1)$ and $n_1, n_2 \in \N$, consider
\begin{equation*}
    r = \left\lceil 8 \log(n_1 + n_2 + 1) \left(1 + \frac{2(C_u + C_v + 1)}{\varepsilon}\right)^2 \right\rceil,
\end{equation*}
where $C_u$, $C_v$ are constants that depend on the LVM. Then for every $\Matrix{F} \in \Real^{n_1 \times n_2}$ generated with $f$, there exists $\Matrix{G} \in \Real^{n_1 \times n_2}$ of $\rank{\Matrix{G}} \leq r$ such that $\Norm{\Matrix{F} - \Matrix{G}}{\max} \leq \varepsilon \Norm{f}{\Ell_\infty(\ball_m^2)}$.
\end{theorem}

The first step in the proof of \cref{theorem:ut} in \cite{udell2019big} consists in building an accurate factorization of the matrix $\Matrix{F}$ based on the Taylor series of the generating function $f$. The following lemma essentially repeats \cref{example:multi_dimensions} with a looser analytical rank bound and provides an estimate of the $\ell_2 \to \ell_\infty$ norms of the factors.\footnote{In the formulation of \cref{lemma:ut} in \cite{udell2019big}, the estimates are written without the squares as $\Norm{\Matrix{u}_i}{2} \leq C_u \Norm{f}{\Ell_\infty}$ and $\Norm{\Matrix{v}_j}{2} \leq C_v \Norm{f}{\Ell_\infty}$, while the ones proved are $\Norm{\Matrix{u}_i}{2}^2 \leq C_u \Norm{f}{\Ell_\infty}$ and $\Norm{\Matrix{v}_j}{2}^2 \leq C_v \Norm{f}{\Ell_\infty}$.}

\begin{lemma}[{\cite[Lemma~4.3]{udell2019big}}]
\label{lemma:ut}
Let $f : \ball_m^2 \to \Real$ satisfy the assumptions of \Cref{example:multi_dimensions} and $\varepsilon \in (0,1)$. Then for every $\Matrix{F} \in \Real^{n_1 \times n_2}$ generated with $f$, there exists $\Matrix{G} \in \Real^{n_1 \times n_2}$ of $\rank{\Matrix{G}} \leq (\rho+1) m^\rho$ with $\rho \leq \max\{ 2emM, \log_2(C/\varepsilon) \} + 1$ such that 
\begin{equation*}
    \Norm{\Matrix{F} - \Matrix{G}}{\max} \leq \varepsilon \Norm{f}{\Ell_\infty(\ball_m^2)}.
\end{equation*}
Furthermore, $\Matrix{G}$ admits a rank-$k$ factorization with $k \leq (\rho+1) m^\rho$ as
\begin{equation*}
    \Matrix{G}(i,j) = \Matrix{u}_i^\trans \Matrix{v}_j, \quad i \in [n_1], \quad j \in [n_2],
\end{equation*}
where each $\Matrix{u}_i \in \Real^{k}$ and $\Matrix{v}_j \in \Real^k$ obeys $\Norm{\Matrix{u}_i}{2}^2 \leq C_u \Norm{f}{\Ell_\infty(\ball_m^2)}$ and $\Norm{\Matrix{v}_j}{2}^2 \leq C_v \Norm{f}{\Ell_\infty(\ball_m^2)}$. Here, $C_u$ and $C_v$ are constants depending on the LVM but not on the dimensions $n_1$ or $n_2$.
\end{lemma}

An analogue of \cref{theorem:compression} is then applied to the factorization of $\Matrix{G}$ to finish the proof of \cref{theorem:ut}. This proof is mathematically correct; however, throughout the text of \cite{udell2019big}, the authors make a number of informal statements that are not in accordance with what they proved. In particular, it is stated that the rank bound of \cref{theorem:ut} is independent of $m$:
\begin{equation}
\tag{Claim A}\label{eq:claima}
\parbox{\dimexpr\linewidth-8em}{%
    ``Our main result, Theorem~4.2 eliminates the dependence on the dimension $m$ of the latent variables by introducing a dependence on the dimension of the matrix.''(\cite[p.148]{udell2019big})
}
\end{equation}
\ref{eq:claima} would follow from \cref{theorem:ut} if both constants $C_u$ and $C_v$ did not depend on $m$, but they do: studying the proof of \cref{lemma:ut} in \cite{udell2019big}, we discover that they are bounded by $C_u > C^2 m^m$ and $C_v > m^m$ \cite[p.154]{udell2019big}. In \Cref{theorem:ut_tighter} below, we manage to make the dependence on $m$ exponential,\footnote{This refers to the explicit dependence. The bound also depends on $m$ implicitly through $C$ and $M$.} but not eliminate it. \ref{eq:claima} turns out to be false.

\begin{theorem}
\label{theorem:ut_tighter}
Let $f : \ball_m^2 \to \Real$ satisfy the assumptions of \cref{example:multi_dimensions}. For all $\varepsilon \in (0,1)$ and $n_1, n_2 \in \N$, consider
\begin{equation*}
    r = \left\lceil 9 \log(3 n_1 n_2) \frac{C^2}{\varepsilon^2} e^{m(1 + M^2)} \right\rceil.
\end{equation*}
Then for every $\Matrix{F} \in \Real^{n_1 \times n_2}$ generated with $f$, there exists $\Matrix{G} \in \Real^{n_1 \times n_2}$ of $\rank{\Matrix{G}} \leq r$ such that $\Norm{\Matrix{F} - \Matrix{G}}{\max} \leq \varepsilon \Norm{f}{\Ell_\infty(\ball_m^2)}$.
\end{theorem}
\begin{proof}
See \Cref{appendix:proof_ut}.
\end{proof}

The authors then discuss the implications of \cref{theorem:ut}:
\begin{gather}
\tag{Claim B}\label{eq:claimb}
\parbox{\dimexpr\linewidth-8em}{%
    ``This result [\cref{theorem:ut}] has ramifications for how to interpret an underlying low rank structure in datasets.~\ldots~Our main theorem shows that low rank structure can persist even without an underlying physical reason. In particular, a dataset from a nice latent variable model has an $\varepsilon$-rank~\ldots~that grows slowly with its dimensions, \emph{no matter how many} [$m$] genres or topics generate the data.'' (emphasis in the original, \cite[pp.145--146]{udell2019big})
} \\
\tag{Claim C}\label{eq:claimc}
\parbox{\dimexpr\linewidth-8em}{%
    ``Our theory shows that a matrix generated from a nice LVM is often well approximated by a matrix of low rank, even if the true latent structure is high dimensional or nonlinear.'' (\cite[p.157]{udell2019big})
}
\end{gather} 
The formulation of \ref{eq:claimb} is open to misinterpretation. If $m$ is assumed to be fixed, \cref{theorem:ut} does indeed provide a bound on the $\varepsilon$-rank that grows as $\mathcal{O}(\log(n))$, with $n = n_1 = n_2$ for simplicity. In this case, the analytical rank bounds \cref{eq:analytical_bound,lemma:ut} guarantee an even stronger property: the $\varepsilon$-rank behaves like $\mathcal{O}(1)$ as $n \to \infty$. However, \ref{eq:claimb} is false as a statement that holds uniformly with respect to varying $m$, due to the $m^m$-dependence of the rank bound in \cref{theorem:ut}. Similarly, \ref{eq:claimc} is correct in the sense that \cref{theorem:ut} provides a bound on the $\varepsilon$-rank that is smaller than $n$ when $n$ is sufficiently large, but the same conclusion follows directly from the analytical rank bounds \cref{eq:analytical_bound,lemma:ut}.

The results of the analytical-algebraic approach of \cite{udell2019big} are inferior to the purely analytical rank bound \cref{eq:analytical_bound} when $n$ is large, but could be superior for moderate $n$. The inadequacy of the analytical rank bound is highlighted multiple times in \cite{udell2019big}, e.g.,
\begin{gather*}
\parbox{\dimexpr\linewidth-2em}{%
    ``\ldots the vectors $\Matrix{u}_i$ and $\Matrix{v}_j$ [in \cref{lemma:ut}] may have an extremely large number of entries when the dimension $m$ \ldots~is large: this bound \ldots~grows as $m^m$.'' (\cite[p.152]{udell2019big})
}
\end{gather*}
On \cite[p.157]{udell2019big}, the bound of \cref{lemma:ut} is presented as $\epsrank{\Matrix{F}} = \mathcal{O}(m^m \log(1/\varepsilon))$ in contrast to $\epsrank{\Matrix{F}} = \mathcal{O}(\log(n) / \varepsilon^2)$ of \cref{theorem:ut}, where $m$ is omitted. The authors conclude in relation to the behavior visualized in \cref{fig:dataset}:
\begin{equation*}
\parbox{\dimexpr\linewidth-2em}{%
    ``Based on these bounds, we should expect that when $m$ is large, then for sufficiently large $n$, $\epsrank{\Matrix{F}}$ grows like $\log(n)$. On the other hand, for small $n$ or $\varepsilon$, we can have $\log(n)/\varepsilon^2 \gtrsim  n$, and hence we may see that $\epsrank{\Matrix{F}}$ grows linearly with $n$.'' (\cite[p.157]{udell2019big})
}
\end{equation*}
This description does not take into account the constant stage in \cref{fig:dataset}, which is inevitable for $n$ large enough. With \ref{eq:claima} disproven, the rank bound of \cref{theorem:ut} needs to be specifically analyzed to determine the conditions under which it is tighter than the analytical bounds (that is, when the logarithmic stage in \cref{fig:dataset} is actually present). Such analysis is absent in \cite{udell2019big}, and we carry it out in \Cref{appendix:worse_bound}. We compare \cref{theorem:ut,theorem:ut_tighter} against the analytical rank bound \cref{eq:analytical_bound} and show that
\begin{itemize}
    \item \cref{theorem:ut} is worse than \cref{eq:analytical_bound} for every $n$ and every $m \geq 2$ when $M < (m^2 - e)/e^3$,
    \item \cref{theorem:ut_tighter} is worse than \cref{eq:analytical_bound} for every $n$ and every $m$ when $M > 1.597$.
\end{itemize}
When $1.597 < M < (m^2 - e)/e^3$, the analytical-algebraic approach of \cite{udell2019big} does not guarantee the existence of the logarithmic stage in \cref{fig:dataset}: it leads to rank bounds that are looser than the analytical rank bound \cref{eq:analytical_bound} for all $n$.

\begin{remark}
The analysis in \Cref{appendix:worse_bound} uses the lower bound $C \geq 1$, which holds for every function from \Cref{example:multi_dimensions}. The Gaussian kernel $f(\Point{x},\Point{y}) = \exp(-\Norm{\Point{x}-\Point{y}}{2}^2)$ has $C = m 4^m$ and $M = 4$ according to \cite[p.151]{udell2019big}. For this $C$, the analytical rank bound \cref{eq:analytical_bound} is tighter for all $n$ when $0 < M < (64m^2 - e)/e^3$. The value $M = 4$ lies in this interval for all $m \geq 2$.
\end{remark}

%% file: parts/dotp.tex
\section{Functions of the inner product}
\label{sec:prod}
The theory developed in \cite{udell2019big} does not guarantee the logarithmic growth of $\varepsilon$-ranks for the Gaussian kernel. Nevertheless, the numerical experiment of \cite{udell2019big} shows that matrices generated with the Gaussian kernel for $m = 1000$ can be approximated with slowly growing ranks. In this section, we provide a theoretical explanation of the observed behavior.

In this numerical experiment, the points $\{ \Point{x}_i \}_{i = 1}^{n_1}$ and $\{ \Point{y}_j \}_{j = 1}^{n_2}$ are drawn independently at random from the uniform distribution on the $m$-dimensional unit sphere. As a consequence, the Gaussian kernel can be rewritten as $f(\Point{x},\Point{y}) = \exp(-2 + 2 \Point{x}^\trans \Point{y})$---this is a function of the inner product. Let us focus on this property and consider functions $f(\Point{x},\Point{y}) = h(\Point{x}^\trans \Point{y})$, where $h$ depends on a single scalar variable. It follows that we can represent matrices $\Matrix{F}$ generated with the function $f$ as $\Matrix{F} = h(\Matrix{X} \Matrix{Y}^\trans)$, where $h$ is applied to $\Matrix{X} \Matrix{Y}^\trans$ entrywise. Here, the rows of $\Matrix{X} \in \Real^{n_1 \times m}$ and $\Matrix{Y} \in \Real^{n_2 \times m}$ are the $m$-dimensional sampling points:
\begin{equation*}
    \Matrix{X}^\trans = 
    \begin{bmatrix}
        \Point{x}_1 & \cdots & \Point{x}_{n_1}    
    \end{bmatrix}, \quad 
    \Matrix{Y}^\trans = 
    \begin{bmatrix}
        \Point{y}_1 & \cdots & \Point{y}_{n_2}    
    \end{bmatrix}.
\end{equation*}

We begin the analysis by extending \cref{theorem:compression} to the compression of the Hadamard product (i.e., entrywise product) of several factorized matrices.

\begin{lemma}
\label{lemma:compression_hadamard}
Let $\varepsilon \in (0,1)$, $n_1, n_2 \in \N$, and $r \in \N$ be given by \cref{eq:compression_rank}. For every $t \in \N$, any $m_1, \ldots, m_t \in \N$, and every collection of matrices $\{ \Matrix{A}_s \}_{s = 1}^{t}$ and $\{ \Matrix{B}_s \}_{s = 1}^{t}$ of sizes $\Matrix{A}_s \in \Real^{n_1 \times m_s}$ and $\Matrix{B}_s \in \Real^{n_2 \times m_s}$, there exists $\Matrix{G} \in \Real^{n_1 \times n_2}$ of $\rank{\Matrix{G}} \leq r$ such that
\begin{equation*}
    \Norm{(\Matrix{A}_1 \Matrix{B}_1^\trans) \hap \cdots \hap (\Matrix{A}_t \Matrix{B}_t^\trans) - \Matrix{G}}{\max} \leq \varepsilon \prod_{s = 1}^{t} \Norm{\Matrix{A}_s}{2,\infty} \Norm{\Matrix{B}_s}{2,\infty}.
\end{equation*}
\end{lemma}
\begin{proof}
By the properties of the Hadamard product \cite[\S 2.6]{kolda2009tensor}, we can write
\begin{equation*}
    (\Matrix{A}_1 \Matrix{B}_1^\trans) \hap \cdots \hap (\Matrix{A}_t \Matrix{B}_t^\trans) = \Tilde{\Matrix{A}} \Tilde{\Matrix{B}}^\trans,
\end{equation*}
where $\Tilde{\Matrix{A}} \in \Real^{n_1 \times (m_1 \ldots m_t)}$ and $\Tilde{\Matrix{B}} \in \Real^{n_2 \times (m_1 \ldots m_t)}$ are defined via the Khatri--Rao product \cref{eq:khrap}:
\begin{equation*}
    \Tilde{\Matrix{A}}^\trans = \Matrix{A}_1^\trans \khrap \cdots \khrap \Matrix{A}_t^\trans, \quad \Tilde{\Matrix{B}}^\trans = \Matrix{B}_1^\trans \khrap \cdots \khrap \Matrix{B}_t^\trans.
\end{equation*}
Applying \cref{theorem:compression}, we obtain the following error bound of rank-$r$ approximation:
\begin{equation*}
    \Norm{\Tilde{\Matrix{A}} \Tilde{\Matrix{B}}^\trans - \Matrix{G}}{\max} \leq \varepsilon \Norm{\Tilde{\Matrix{A}}}{2,\infty} \Norm{\Tilde{\Matrix{B}}}{2,\infty}.
\end{equation*}
The rows of $\Tilde{\Matrix{A}}$ and $\Tilde{\Matrix{B}}$ are Kronecker products of the corresponding rows of the initial matrices,
\begin{equation*}
    \Tilde{\Matrix{A}}(i,:) = \Matrix{A}_1(i,:) \krp \cdots \krp \Matrix{A}_t(i,:), \quad \Tilde{\Matrix{B}}(j,:) = \Matrix{B}_1(j,:) \krp \cdots \krp \Matrix{B}_t(j,:),
\end{equation*}
so that $\Norm{\Tilde{\Matrix{A}}(i,:)}{2} = \prod_{s = 1}^{t} \Norm{\Matrix{A}_s(i,:)}{2}$ and $\Norm{\Tilde{\Matrix{B}}(j,:)}{2} = \prod_{s = 1}^{t} \Norm{\Matrix{B}_s(j,:)}{2}$.
\end{proof}

In the following lemma, we show how the particular structure of the generating function $f(\Point{x},\Point{y}) = h(\Point{x}^\trans \Point{y})$ paves the way for the low-rank approximation of the corresponding function-generated matrices with rank that is independent of the dimension $m$.

\begin{assumption}
\label{assumption:1d_smooth}
$h : \Omega \to \Real$ is in $C^{\infty}(\Omega)$ for a sufficiently large interval $\Omega \subset \Real$.
\end{assumption}

\begin{lemma}
\label{lemma:1d_function_dotp}
Let $\varepsilon \in (0,1)$, $n_1, n_2 \in \N$, and $r \in \N$ be given by \cref{eq:compression_rank}. Let $h : \Omega \to \Real$ satisfy \cref{assumption:1d_smooth} with $0 \in \Omega$. For any $t,m \in \N$ and every pair of $\Matrix{A} \in \Real^{n_1 \times m}$ and $\Matrix{B} \in \Real^{n_2 \times m}$ such that $\mathcal{I} = [\min\{\Matrix{A} \Matrix{B}^\trans\}, \max\{\Matrix{A} \Matrix{B}^\trans\}] \subseteq \Omega$, there exists $\Matrix{G} \in \Real^{n_1 \times n_2}$ of $\rank{\Matrix{G}} \leq 1 + (t-1)r$ such that
\begin{equation*}
    \Norm{h(\Matrix{A} \Matrix{B}^\trans) - \Matrix{G}}{\max} \leq \frac{\Norm{\Matrix{A} \Matrix{B}^\trans}{\max}^t}{t!} \Norm{h^{(t)}}{\Ell_\infty(\Omega)} + \varepsilon \sum_{s = 1}^{t-1} \frac{|h^{(s)}(0)|}{s!} \Norm{\Matrix{A}}{2,\infty}^s \Norm{\Matrix{B}}{2,\infty}^s.
\end{equation*}
\end{lemma}
\begin{proof}
The truncated Taylor series $h_t(x) = \sum_{s = 0}^{t-1} \frac{h^{(s)}(0)}{s!} x^s$ approximates $h$ as
\begin{equation*}
    \Norm{h - h_t}{\Ell_\infty(\mathcal{I})} \leq \frac{\Norm{\Matrix{A} \Matrix{B}^\trans}{\max}^t}{t!} \Norm{h^{(t)}}{\Ell_\infty(\Omega)}.
\end{equation*}
Note that we cannot use $\Norm{h^{(t)}}{\Ell_\infty(\mathcal{I})}$ in the right-hand side since the origin might not belong to the interval $\mathcal{I}$. Next, we evaluate $h_t(\Matrix{A} \Matrix{B}^\trans) = \sum_{s = 0}^{t-1} \frac{h^{(s)}(0)}{s!} (\Matrix{A} \Matrix{B}^\trans)^{\hap s}$. For every $s \in [t-1]$, we can apply \cref{lemma:compression_hadamard} to $(\Matrix{A} \Matrix{B}^\trans)^{\hap s}$ to find its approximation of rank at most $r$. For $s = 0$, the corresponding term is a matrix of $h(0)$, hence its rank is at most one. Finally, recall that the rank of a sum of two matrices is bounded by the sum of their ranks \cite[\S0.4]{horn2012matrix}.
\end{proof}

\begin{remark}
We can relax \cref{assumption:1d_smooth}, requiring $h \in C^q(\Omega)$ and taking $t \leq q$.
\end{remark}

The error bound in \cref{lemma:1d_function_dotp} consists of two terms. The first is ``analytical'': it is responsible for discretizing the generating function and setting the baseline level of the achievable approximation error. The second is ``algebraic'': it adds up the contributions from the approximation of the selected monomials and controls the target approximation error. Next, we request that the derivatives of $h$ satisfy the bound of \cref{example:1d} to achieve the exponential decay of the baseline error.

\begin{assumption}
\label{assumption:1d_analytic}
$h : \Omega \to \Real$ satisfies \cref{assumption:1d_smooth}, is bounded in $\Omega$, and its derivatives grow as $\Norm{h^{(s)}}{\Ell_\infty(\Omega)} \leq C M^s$ for $s \in \N_0$ with constants $C = C_{\Omega, h} \geq 0$ and $M = M_{\Omega, h} > 0$.
\end{assumption}

\begin{remark}
Compared to \cref{example:1d,example:multi_dimensions}, we absorb the norm $\Norm{h}{\Ell_\infty(\Omega)}$ into $C$. Note also that, unlike in \cref{example:multi_dimensions}, the constants $C$ and $M$ do not depend on $m$.
\end{remark}

\begin{corollary}
\label{corollary:1d_function_dotp}
In the setting of \cref{lemma:1d_function_dotp}, let $h : \Omega \to \Real$ satisfy \cref{assumption:1d_analytic}. If $\max\{\Norm{\Matrix{A}}{2,\infty}, \Norm{\Matrix{B}}{2,\infty}\} \leq R$ then, for $t \geq e^2 M R^2$, 
\begin{equation*}
    \Norm{h(\Matrix{A} \Matrix{B}^\trans) - \Matrix{G}}{\max} \leq C \Big[ e^{-t} + \varepsilon (e^{M R^2} - 1) \Big].
\end{equation*}
\end{corollary}
\begin{proof}
Use Stirling's formula (as in \cref{example:1d}) in the error bound of \cref{lemma:1d_function_dotp}, estimate the sum over $s$ with the infinite series, and note that $\Norm{\Matrix{A} \Matrix{B}^\trans}{\max} \leq R^2$.
\end{proof}

\begin{theorem}
\label{theorem:1d_function_dotp_gen}
Let $\varepsilon \in (0,1)$, $n_1, n_2 \in \N$, and $r \in \N$ be given by \cref{eq:compression_rank}. Let $h : \Omega \to \Real$ satisfy \cref{assumption:1d_analytic} with $[-\sigma R^2, \sigma R^2] \subseteq \Omega$ and $\sigma, R > 0$. Let $m \in \N$ and $\Matrix{W} \in \Real^{m \times m}$ be symmetric positive-definite with $\Norm{\Matrix{W}}{2} \leq \sigma$. For every $\Matrix{F} \in \Real^{n_1 \times n_2}$ generated with $f(\Point{x},\Point{y}) = h(\Point{x}^\trans \Matrix{W} \Point{y})$ based on points from $\overline{\ball}_{m,R}$ and every integer $t \geq e^2 M \Norm{\Matrix{W}}{2} R^2$, there exists $\Matrix{G} \in \Real^{n_1 \times n_2}$ of $\rank{\Matrix{G}} \leq 1 + (t-1)r$ such that
\begin{equation*}
    \Norm{\Matrix{F} - \Matrix{G}}{\max} \leq C \Big[ e^{-t} + \varepsilon (e^{M \Norm{\Matrix{W}}{2} R^2} - 1) \Big].
\end{equation*}
\end{theorem}
\begin{proof}
Let the rows of $\Matrix{X} \in \Real^{n_1 \times m}$ and $\Matrix{Y} \in \Real^{n_2 \times m}$ consist of the points used to generate $\Matrix{F}$ so that $\Matrix{F} = h(\Matrix{X} \Matrix{W} \Matrix{Y}^\trans)$. Let $\Matrix{W} = \Matrix{U} \Matrix{\Lambda} \Matrix{U}^\trans$ be the eigenvalue decomposition of $\Matrix{W}$ with orthogonal $\Matrix{U}$ and diagonal positive-definite $\Matrix{\Lambda}$. Note that $\Matrix{X} \Matrix{W} \Matrix{Y}^\trans = (\Matrix{X}\Matrix{U} \Matrix{\Lambda}^{\frac{1}{2}}) (\Matrix{Y}\Matrix{U} \Matrix{\Lambda}^{\frac{1}{2}})^\trans$ and
\begin{equation*}
    \Norm{\Matrix{X}\Matrix{U} \Matrix{\Lambda}^{\frac{1}{2}}}{2,\infty} \leq \Norm{\Matrix{\Lambda}^{\frac{1}{2}}}{2} \Norm{\Matrix{X}}{2,\infty} \leq \sqrt{\Norm{\Matrix{W}}{2}} R, \quad     \Norm{\Matrix{Y}\Matrix{U} \Matrix{\Lambda}^{\frac{1}{2}}}{2,\infty} \leq \Norm{\Matrix{\Lambda}^{\frac{1}{2}}}{2} \Norm{\Matrix{Y}}{2,\infty} \leq \sqrt{\Norm{\Matrix{W}}{2}} R.
\end{equation*}
It remains to apply \cref{corollary:1d_function_dotp} with $\Matrix{A} = \Matrix{X}\Matrix{U} \Matrix{\Lambda}^{\frac{1}{2}}$ and $\Matrix{B} = \Matrix{Y}\Matrix{U} \Matrix{\Lambda}^{\frac{1}{2}}$.
\end{proof}

\begin{remark}
The statement holds for general non-symmetric rectangular matrices $\Matrix{W}$. In the proof, the eigenvalue decomposition needs to be replaced with the singular value decomposition. And, of course, the corresponding bilinear form will no longer be an inner product.
\end{remark}

For sufficiently small $\varepsilon$, the choice $t = \lceil \log(\varepsilon^{-1}) \rceil$ guarantees that there is a matrix $\Matrix{G}$ of $\rank{\Matrix{G}} = \mathcal{O}(\log(n_1 n_2) \varepsilon^{-2} \log(\varepsilon^{-1}))$ such that $\Norm{\Matrix{F} - \Matrix{G}}{\max} = \mathcal{O}(\varepsilon)$. As we noted before, neither of the constants in the rank bound and the error bound depends on $m$.

\begin{remark}
\label{remark:gaussian}
Let us return to the Gaussian kernel $f(\Point{x},\Point{y}) = \exp(-\Norm{\Point{x}-\Point{y}}{2}^2)$ on the unit sphere. It corresponds to $h(x) = \exp(2x - 2)$ with $x \in [-1,1]$. This $h$ satisfies \Cref{assumption:1d_analytic} with $C = 1$ and $M = 2$. By \Cref{theorem:1d_function_dotp_gen}, every $\Matrix{F} \in \Real^{n_1 \times n_2}$ generated with $f$ can be approximated as $\Norm{\Matrix{F} - \Matrix{G}}{\max} \leq e^2\varepsilon$ with $\rank{\Matrix{G}} \leq \lceil 9 \log(3 n_1 n_2) \varepsilon^{-2} \log(\varepsilon^{-1}) \rceil$. The conditions of \Cref{theorem:1d_function_dotp_gen} require that $\lceil \log(\varepsilon^{-1}) \rceil \geq 2 e^2$, which holds for $\varepsilon \leq \exp(-2 e^2) < 3.82 \times 10^{-7}$.
\end{remark}

%% file: parts/dist.tex
\section{Functions of the squared Euclidean distance}
\label{sec:dist}
When restricted to the unit sphere, the Gaussian kernel $f(\Point{x},\Point{y}) = \exp(-\Norm{\Point{x}-\Point{y}}{2}^2)$ depends only on the inner product $\Point{x}^\trans \Point{y}$. Should the variables belong to the unit ball, $f$ ceases to be a function of the inner product since $f(\Point{x},\Point{y}) = \exp(-\Norm{\Point{x}}{2}^2)\exp(-\Norm{\Point{y}}{2}^2)\exp(2 \Point{x}^\trans \Point{y})$. A matrix generated with a product of functions is the Hadamard product of the corresponding function-generated matrices; in our case, $\Matrix{F} = \Matrix{F}_1 \hap \Matrix{F}_2 \hap \Matrix{F}_3$ with $f_1(\Point{x},\Point{y}) = \exp(-\Norm{\Point{x}}{2}^2)$, $f_2(\Point{x},\Point{y}) = \exp(-\Norm{\Point{y}}{2}^2)$, and $f_3(\Point{x},\Point{y}) = \exp(2 \Point{x}^\trans \Point{y})$. The first two functions depend on a single variable, so $\Matrix{F}_1$ and $\Matrix{F}_2$ are rank-one matrices, and $\Matrix{F}_3$ can be approximated with $\Matrix{G}$ of $\rank{\Matrix{G}} = \mathcal{O}(\log(n_1 n_2) \varepsilon^{-2} \log(\varepsilon^{-1}))$ by \cref{theorem:1d_function_dotp_gen}. By the properties of the Hadamard product \cite[\S 5.1]{horn1994topics},
\begin{gather*}
    \rank{\Matrix{F}_1 \hap \Matrix{F}_2 \hap \Matrix{G}} \leq \rank{\Matrix{F}_1} \rank{\Matrix{F}_2} \rank{\Matrix{G}} = \rank{\Matrix{G}}, \\
    \Norm{\Matrix{F} - \Matrix{F}_1 \hap \Matrix{F}_2 \hap \Matrix{G}}{\max} \leq \Norm{\Matrix{F}_1}{\max} \Norm{\Matrix{F}_2}{\max} \Norm{\Matrix{F}_3 - \Matrix{G}}{\max} = \mathcal{O}(\varepsilon).    
\end{gather*}
For matrices generated with Gaussian kernels, it does not matter whether the variables are sampled from a sphere or from a ball. In proving this, we relied on the properties of the exponential. What happens in the more general case of $f(\Point{x},\Point{y}) = h(\Norm{\Point{x}-\Point{y}}{2}^2)$?

Analytical low-rank approximation of such functions was studied in \cite{wang2018numerical}. As follows from its results, the entrywise error $\Norm{\Matrix{F} - \Matrix{G}}{\max} = \mathcal{O}(\varepsilon)$ can be achieved with a matrix $\Matrix{G}$ of $\rank{\Matrix{G}} \leq \genfrac{(}{)}{0pt}{1}{m + 2 + \rho_\ast}{m + 2}$ with $\rho_\ast = \mathcal{O}(\log(\varepsilon^{-1}))$. This bound qualitatively coincides with \cref{eq:analytical_bound} for fixed $m$ and $\varepsilon \to 0$, and is an improvement for fixed $\varepsilon$ and $m \to \infty$ since $\genfrac{(}{)}{0pt}{1}{m + 2 + \rho_\ast}{m + 2} \sim \frac{(m+2)^{\rho_\ast}}{\rho_\ast!}$.

For fixed $\varepsilon$, the above analytical rank bound grows polynomially with the dimension $m$. In this section, we show how to eliminate the dependence on $m$ whatsoever. We begin with an auxiliary lemma that will allow us to reuse the results from \Cref{sec:prod}.

\begin{lemma}
\label{lemma:prod_plus_factorization}
For any $n_1,n_2,m,p \in \N$, every pair of $\Matrix{A} \in \Real^{n_1 \times m}$ and $\Matrix{B} \in \Real^{n_2 \times m}$, and any rank-one matrices $\Matrix{Z}_1, \ldots, \Matrix{Z}_p \in \Real^{n_1 \times n_2}$, there exist $\Matrix{C} \in \Real^{n_1 \times (m+p)}$ and $\Matrix{D} \in \Real^{n_2 \times (m+p)}$ such that $\Matrix{A} \Matrix{B}^\trans + \sum_{i = 1}^{p} \Matrix{Z}_i = \Matrix{C} \Matrix{D}^\trans$ and
\begin{equation*}
    \Norm{\Matrix{C}}{2,\infty} \leq \bigg( \Norm{\Matrix{A}}{2,\infty}^2 + \sum_{i = 1}^{p} \Norm{\Matrix{Z}_i}{\max} \bigg)^{\frac{1}{2}}, \quad \Norm{\Matrix{D}}{2,\infty} \leq \bigg( \Norm{\Matrix{B}}{2,\infty}^2 + \sum_{i = 1}^{p} \Norm{\Matrix{Z}_i}{\max} \bigg)^{\frac{1}{2}}.
\end{equation*}
\end{lemma}
\begin{proof}
We can represent every $\Matrix{Z}_i$ as $\Matrix{Z}_i = \Matrix{u}_i \Matrix{v}_i^\trans$ with $\Matrix{u}_i \in \Real^{n_1}$ and $\Matrix{v}_i \in \Real^{n_2}$ normalized so that $\Norm{\Matrix{u}_i}{\infty} = \Norm{\Matrix{v}_i}{\infty} = \sqrt{\Norm{\Matrix{Z}_i}{\max}}$. Consider $\Matrix{C} \in \Real^{n_1 \times (m+p)}$ and $\Matrix{D} \in \Real^{n_2 \times (m+p)}$ given by
\begin{equation*}
    \Matrix{C} = \begin{bmatrix}
        \Matrix{A} & \Matrix{u}_1 & \cdots & \Matrix{u}_p
    \end{bmatrix}, \quad
    \Matrix{D} = \begin{bmatrix}
        \Matrix{B} & \Matrix{v}_1 & \cdots & \Matrix{v}_p
    \end{bmatrix}
\end{equation*}
such that $\Matrix{A} \Matrix{B}^\trans + \sum_{i = 1}^{p} \Matrix{Z}_i = \Matrix{C} \Matrix{D}^\trans$. Let us estimate $\Norm{\Matrix{C}}{2,\infty}$ (a similar bound holds for $\Norm{\Matrix{D}}{2,\infty}$):
\begin{equation*}
    \Norm{\Matrix{C}}{2,\infty}^2 = \max_j \bigg\{\Norm{\Matrix{A}(j,:)}{2}^2 + \sum_{i = 1}^{p} |\Matrix{u}_i(j)|^2\bigg\} \leq \Norm{\Matrix{A}}{2,\infty}^2 + \sum_{i = 1}^{p} \Norm{\Matrix{u}_i}{\infty}^2.
\end{equation*}
\end{proof}

\cref{corollary:1d_function_dotp} estimates the low-rank approximation error for $h(\Matrix{A} \Matrix{B}^\trans)$ in terms of $\Norm{\Matrix{A}}{2,\infty}$ and $\Norm{\Matrix{B}}{2,\infty}$. If in addition to these two norms we can also control $\Norm{\Matrix{Z}_i}{\max}$, \cref{corollary:1d_function_dotp} can be applied to $h(\Matrix{A} \Matrix{B}^\trans + \sum_{i = 1}^{p} \Matrix{Z}_i)$ with the help of \cref{lemma:prod_plus_factorization}.

\begin{theorem}
\label{theorem:1d_function_dist_gen}
Let $\varepsilon \in (0,1)$, $n_1, n_2 \in \N$, and $r \in \N$ be given by \cref{eq:compression_rank}. Let $h : \Omega \to \Real$ satisfy \cref{assumption:1d_analytic} with $[-4\sigma R^2, 4\sigma R^2] \subseteq \Omega$ and $\sigma, R > 0$. Let $m \in \N$ and $\Matrix{W}~\in~\Real^{m \times m}$ be symmetric positive-definite with $\Norm{\Matrix{W}}{2} \leq \sigma$. For every $\Matrix{F} \in \Real^{n_1 \times n_2}$ generated with $f(\Point{x},\Point{y}) = h(\Norm{\Point{x} - \Point{y}}{\Matrix{W}}^2)$ based on points from $\overline{\ball}_{m,R}$ and every integer $t~\geq~4e^2 M \Norm{\Matrix{W}}{2} R^2$, there exists $\Matrix{G} \in \Real^{n_1 \times n_2}$ of $\rank{\Matrix{G}} \leq 1 + (t-1)r$ such that
\begin{equation*}
    \Norm{\Matrix{F} - \Matrix{G}}{\max} \leq C \Big[ e^{-t} + \varepsilon (e^{4M \Norm{\Matrix{W}}{2} R^2} - 1) \Big].
\end{equation*}
\end{theorem}
\begin{proof}
Exactly as in the proof of \cref{theorem:1d_function_dotp_gen}, we introduce $\Matrix{X} \in \Real^{n_1 \times m}$ and $\Matrix{Y} \in \Real^{n_2 \times m}$ that consist of the sampling points and consider the eigenvalue decomposition $\Matrix{W} = \Matrix{U} \Matrix{\Lambda} \Matrix{U}^\trans$. Note that $\Matrix{F} = h(\Matrix{Z}_1 + \Matrix{Z}_2 - 2 \Matrix{X} \Matrix{W} \Matrix{Y}^\trans)$, where $\Matrix{Z}_1, \Matrix{Z}_2 \in \Real^{n_1 \times n_2}$ are rank-one matrices
\begin{equation*}
    \Matrix{Z}_1 = 
    \begin{bmatrix}
        \Norm{\Point{x}_1}{\Matrix{W}}^2 & \cdots & \Norm{\Point{x}_1}{\Matrix{W}}^2 \\
        \vdots & \ddots & \vdots \\
        \Norm{\Point{x}_{n_1}}{\Matrix{W}}^2 & \cdots & \Norm{\Point{x}_{n_1}}{\Matrix{W}}^2 \\
    \end{bmatrix}, \quad 
    \Matrix{Z}_2 = 
    \begin{bmatrix}
        \Norm{\Point{y}_1}{\Matrix{W}}^2 & \cdots & \Norm{\Point{y}_{n_2}}{\Matrix{W}}^2 \\
        \vdots & \ddots & \vdots \\
        \Norm{\Point{y}_1}{\Matrix{W}}^2 & \cdots & \Norm{\Point{y}_{n_2}}{\Matrix{W}}^2 \
    \end{bmatrix}.
\end{equation*}
Let $\Matrix{A} = -\sqrt{2}\Matrix{X}\Matrix{U} \Matrix{\Lambda}^{\frac{1}{2}}$ and $\Matrix{B} = \sqrt{2}\Matrix{Y}\Matrix{U} \Matrix{\Lambda}^{\frac{1}{2}}$ so that $\Matrix{A} \Matrix{B}^\trans = -2\Matrix{X} \Matrix{W} \Matrix{Y}^\trans$. Then 
\begin{gather*}
    \Norm{\Matrix{A}}{2,\infty} \leq \sqrt{2\Norm{\Matrix{W}}{2}} R, \quad \Norm{\Matrix{B}}{2,\infty} \leq \sqrt{2\Norm{\Matrix{W}}{2}} R, \\
    \Norm{\Matrix{Z}_1}{\max} \leq \Norm{\Matrix{W}}{2} R^2, \quad \Norm{\Matrix{Z}_2}{\max} \leq \Norm{\Matrix{W}}{2} R^2.
\end{gather*}
It remains to apply \cref{lemma:prod_plus_factorization,corollary:1d_function_dotp}.
\end{proof}

\begin{remark}
Just as in \cref{theorem:1d_function_dotp_gen}, the matrix $\Matrix{W}$ need not be positive definite. In particular, the proof does not depend on the signs of $(\Point{x} - \Point{y})^\trans \Matrix{W} (\Point{x} - \Point{y})$.
\end{remark}

When $t = \lceil \log(\varepsilon^{-1}) \rceil$, there exists a matrix $\Matrix{G}$ of $\rank{\Matrix{G}} = \mathcal{O}(\log(n_1 n_2) \varepsilon^{-2} \log(\varepsilon^{-1}))$ such that $\Norm{\Matrix{F} - \Matrix{G}}{\max} = \mathcal{O}(\varepsilon)$, and the hidden constants are independent of the dimension $m$.

\subsection{Localized sample-aware bounds}
\cref{theorem:1d_function_dotp_gen,theorem:1d_function_dist_gen} are based on the Taylor series of the function $h$ at zero and depend on its smoothness there. As a result, \cref{theorem:1d_function_dist_gen} cannot be applied to functions $h$ that are non-smooth at zero, even if all pairwise distances between the sampling points are strictly positive. This technical issue concerns, for example, the exponential kernel $f(\Point{x}, \Point{y}) = \exp(- \Norm{\Point{x} - \Point{y}}{2})$. With the help of \cref{lemma:prod_plus_factorization}, we can overcome this problem by considering Taylor series at arbitrary points.

\begin{theorem}
\label{theorem:1d_local_function_dist}
Let $\varepsilon \in (0,1)$, $n_1, n_2 \in \N$, and $r \in \N$ be given by \cref{eq:compression_rank}. Let $m \in \N$ and $\Matrix{W} \in \Real^{m \times m}$ be symmetric positive-definite. Let $R > 0$, sample points $\{ \Point{x}_i \}_{i = 1}^{n_1} \subset \overline{\ball}_{m,R}$ and $\{ \Point{y}_j \}_{j = 1}^{n_2} \subset \overline{\ball}_{m,R}$, and define $\Matrix{D} \in \Real^{n_1 \times n_2}$ as $\Matrix{D}(i,j) = \Norm{\Point{x}_i - \Point{y}_j}{\Matrix{W}}^2$. Let $h : \Omega \to \Real$ satisfy \cref{assumption:1d_smooth} with $\Omega = [\min\{ \Matrix{D} \}, \max\{ \Matrix{D} \}]$. For every $t \in \N$, there exists $\Matrix{G} \in \Real^{n_1 \times n_2}$ of $\rank{\Matrix{G}} \leq 1 + (t-1)r$ such that
\begin{equation*}
    \Norm{h(\Matrix{D}) - \Matrix{G}}{\max} \leq \frac{(\max\{ \Matrix{D} \} - \min\{ \Matrix{D} \})^t}{2^t \cdot t!} \Norm{h^{(t)}}{\Ell_\infty(\Omega)} + \varepsilon \sum_{s = 1}^{t-1} \frac{|h^{(s)}(\xi)|}{s!} ( 8\Norm{\Matrix{W}}{2} R^2 )^s,
\end{equation*}
where $\xi = (\min\{ \Matrix{D} \} + \max\{ \Matrix{D} \}) / 2$.
\end{theorem}
\begin{proof}
The truncated Taylor series $h_t(x) = \sum_{s = 0}^{t-1} \frac{h^{(s)}(\xi)}{s!} (x - \xi)^s$ of $h$ at $\xi$ satisfies
\begin{equation*}
    \Norm{h - h_t}{\Ell_\infty(\Omega)} \leq \frac{(\max\{ \Matrix{D} \} - \min\{ \Matrix{D} \})^t}{2^t \cdot t!} \Norm{h^{(t)}}{\Ell_\infty(\Omega)}.
\end{equation*}
Next, we evaluate $h_t$ entrywise as $h_t(\Matrix{D}) = \sum_{s = 0}^{t-1} \frac{h^{(s)}(\xi)}{s!} (\Matrix{D} - \xi)^{\hap s}$. Since $\xi$ is a scalar, the entrywise difference $\Matrix{D} - \xi$ is a rank-one perturbation of $\Matrix{D}$, and we can apply \cref{lemma:prod_plus_factorization,lemma:compression_hadamard} to obtain low-rank approximations of its Hadamard powers. Namely, for $s \in [t-1]$, there exists $\Matrix{G}_s \in \Real^{n_1 \times n_2}$ of $\rank{\Matrix{G}_s} \leq r$ such that
\begin{equation*}
    \Norm{(\Matrix{D} - \xi)^{\hap s} - \Matrix{G}_s}{\max} \leq \varepsilon (4\Norm{\Matrix{W}}{2} R^2 + |\xi| )^s \leq \varepsilon ( 8\Norm{\Matrix{W}}{2} R^2 )^s,
\end{equation*}
where we used estimates from the proof of \cref{theorem:1d_function_dist_gen}. Then $\Matrix{G} = h(\xi) + \sum_{s = 1}^{t-1} \frac{h^{(s)}(\xi)}{s!} \Matrix{G}_s$ of $\rank{\Matrix{G}} \leq 1 + (t-1)r$ satisfies
\begin{align*}
    \Norm{h(\Matrix{D}) - \Matrix{G}}{\max} &\leq \Norm{h(\Matrix{D}) - h_t(\Matrix{D})}{\max} + \Norm{h_t(\Matrix{D}) - \Matrix{G}}{\max} \\
    &\leq \frac{(\max\{ \Matrix{D} \} - \min\{ \Matrix{D} \})^t}{2^t \cdot t!} \Norm{h^{(t)}}{\Ell_\infty(\Omega)} + \varepsilon \sum_{s = 1}^{t-1} \frac{|h^{(s)}(\xi)|}{s!} ( 8\Norm{\Matrix{W}}{2} R^2 )^s.
\end{align*}
\end{proof}

\begin{corollary}
\label{corollary:1d_function_dist_gen_shifted}
In the setting of \cref{theorem:1d_local_function_dist}, let $h : \Omega \to \Real$ satisfy \cref{assumption:1d_analytic}. Then, for $t \geq e^2 M (\max\{ \Matrix{D} \} - \min\{ \Matrix{D} \}) / 2$, we have
\begin{equation*}
    \Norm{h(\Matrix{D}) - \Matrix{G}}{\max} \leq C \Big[ e^{-t} + \varepsilon (e^{8 M \Norm{\Matrix{W}}{2} R^2} - 1) \Big].
\end{equation*}
\end{corollary}
\begin{proof}
Use Stirling's formula (as in \cref{example:1d}).
\end{proof}

\cref{corollary:1d_function_dist_gen_shifted} extends the scope of \cref{theorem:1d_function_dist_gen} to functions $h$ that are non-smooth at zero. There are two more improvements. We show that the exponential decay of the ``analytical'' error begins at smaller values of $t$ and narrow the domain $\Omega$ to which the constants $C$ and $M$ correspond. The same modifications can be incorporated in \cref{theorem:1d_function_dotp_gen} too, but they will not have any significant impact unless all of the inner products have the same sign.

%% file: parts/kernels.tex
\section{Positive-definite kernels}
\label{sec:kernels}
Previously, we relied on the smoothness of the generating function to approximate it via the truncated Taylor series. With different analytical assumptions come other approaches to low-rank approximation of functions.

Kernel functions are actively used in machine learning. A function $f : \Omega^2 \to \Real$ is called a \emph{positive-definite} kernel if it is symmetric, $f(\Point{x}, \Point{y}) = f(\Point{y}, \Point{x})$ for all $\Point{x}, \Point{y} \in \Omega$, and for each $n \in \N$, every $n \times n$ symmetric matrix generated with $f$ according to $\Matrix{F}(i,j) = f(\Point{x}_i, \Point{x}_j)$, $\{ \Point{x}_i \}_{i = 1}^{n} \subset \Omega$, is positive-definite. We shall say that such matrices $\Matrix{F}$ are symmetric function-generated.

Many of the kernels used in practice (such as the Gaussian, Cauchy, and exponential kernels) possess an additional property of being \emph{shift-invariant}, by which we mean that there exists a function $\kappa$ such that $f(\Point{x}, \Point{y}) = \kappa(\Point{x} - \Point{y})$. A celebrated result related to shift-invariant kernels is their low-rank approximation via random Fourier features \cite{rahimi2007random}.

\begin{theorem}
\label{theorem:rff}
Let $f: \Real^m \times \Real^m \to \Real$ be a continuous shift-invariant positive-definite kernel with $\kappa(\Point{0}) = 1$. For every $R > 0$ and $\varepsilon \in (0,1)$, there exist $\rho \in \N$ and $\Point{\omega}_1, \ldots, \Point{\omega}_\rho \in \Real^m$ such that $f_\rho(\Point{x}, \Point{y}) = \frac{1}{\rho}\sum\nolimits_{i = 1}^{\rho} \cos(\Point{\omega}_i^T (\Point{x} - \Point{y}))$ satisfies $\Norm{f - f_\rho}{\Ell_{\infty}(\overline{\ball}_{m,R}^2)} \leq \varepsilon$.
\end{theorem}

According to \cite{rahimi2007random}, the frequencies $\{ \Point{\omega}_i \}$ can be chosen at random from a suitable distribution so that the approximation bound is guaranteed with constant probability when $\rho$ is of order $m / \varepsilon^2$, up to logarithmic factors. The rank bound was improved in the case where the kernel is restricted to an algebraic variety by replacing $m$ with the dimension of the variety \cite{altschuler2023kernel}. Using random embeddings, we can obtain rank bounds that are independent of $m$.

\begin{theorem}
\label{theorem:kernel_function}
Let $\varepsilon \in (0,1)$, $n_1, n_2 \in \N$, and $r \in \N$ be given by \cref{eq:compression_rank}. Let $f: \Real^m \times \Real^m \to \Real$ satisfy the assumptions of \cref{theorem:rff}. Then for every $\Matrix{F} \in \Real^{n_1 \times n_2}$ generated with $f$ based on any points from $\Real^m$, there exists $\Matrix{G} \in \Real^{n_1 \times n_2}$ of $\rank{\Matrix{G}} \leq r$ such that $\Norm{\Matrix{F} - \Matrix{G}}{\max} \leq \varepsilon$.
\end{theorem}
\begin{proof}
Let $R > 0$ be such that the sampling points satisfy $\{ \Point{x}_i \}_{i = 1}^{n_1} \subset \overline{\ball}_{m,R}$ and $\{ \Point{y}_j \}_{j = 1}^{n_2} \subset \overline{\ball}_{m,R}$. Fix $\delta > 0$ and consider the function $f_\rho$ from \cref{theorem:rff} such that $\Norm{f - f_\rho}{\Ell_{\infty}(\overline{\ball}_{m,R}^2)} \leq \delta$. Evaluate $f_\rho$ at the sampling points to obtain a matrix that can be factorized as $\Matrix{A} \Matrix{B}^\trans$ with
\begin{equation*}
    \Matrix{A} = \frac{1}{\sqrt{\rho}}
    \begin{bmatrix}
        \cos(\Point{\omega}_1^\trans \Point{x}_1) & \sin(\Point{\omega}_1^\trans \Point{x}_1) & \cdots & \cos(\Point{\omega}_\rho^\trans \Point{x}_1) & 
        \sin(\Point{\omega}_\rho^\trans \Point{x}_1) \\
        \vdots & \vdots & \ddots & \vdots & \vdots \\
        \cos(\Point{\omega}_1^\trans \Point{x}_{n_1}) & \sin(\Point{\omega}_1^\trans \Point{x}_{n_1}) & \cdots & \cos(\Point{\omega}_\rho^\trans \Point{x}_{n_1}) & 
        \sin(\Point{\omega}_\rho^\trans \Point{x}_{n_1}) \\
    \end{bmatrix} \in \Real^{n_1 \times 2\rho},  
\end{equation*}
\begin{equation*}
    \Matrix{B} = \frac{1}{\sqrt{\rho}}
    \begin{bmatrix}
        \cos(\Point{\omega}_1^\trans \Point{y}_1) & \sin(\Point{\omega}_1^\trans \Point{y}_1) & \cdots & \cos(\Point{\omega}_\rho^\trans \Point{y}_1) & 
        \sin(\Point{\omega}_\rho^\trans \Point{y}_1) \\
        \vdots & \vdots & \ddots & \vdots & \vdots \\
        \cos(\Point{\omega}_1^\trans \Point{y}_{n_2}) & \sin(\Point{\omega}_1^\trans \Point{y}_{n_2}) & \cdots & \cos(\Point{\omega}_\rho^\trans \Point{y}_{n_2}) & 
        \sin(\Point{\omega}_\rho^\trans \Point{y}_{n_2}) \\
    \end{bmatrix} \in \Real^{n_2 \times 2\rho}.
\end{equation*}
Clearly, $\Norm{\Matrix{A}}{2,\infty} = \Norm{\Matrix{B}}{2,\infty} = 1$. We can then apply \cref{theorem:compression} to find $\Matrix{G} \in \Real^{n_1 \times n_2}$ of $\rank{\Matrix{G}} \leq r$ such that $\Norm{\Matrix{F} - \Matrix{G}}{\max} \leq \delta + \varepsilon$. Note that $\Matrix{G}$ depends on $\delta$ and take a vanishing sequence $\{ \delta_l \}$. The corresponding sequence of low-rank matrices $\{ \Matrix{G}_l \}$ is bounded by the above inequality, hence it has a subsequence converging to $\Matrix{G}_{\ast} \in \Real^{n_1 \times n_2}$ such that $\Norm{\Matrix{F} - \Matrix{G}_{\ast}}{\max} \leq \varepsilon$ and $\rank{\Matrix{G}_\ast} \leq r$, since the set $\set{\Matrix{G} \in \Real^{n_1 \times n_2}}{\rank{\Matrix{G}} \leq r}$ is closed.
\end{proof}

\begin{remark}
\cref{theorem:kernel_function} is applicable to any (non-symmetric and rectangular) matrix generated with a shift-invariant positive-definite kernel. Similar bounds hold for symmetric function-generated matrices when the kernel is only positive-semidefinite \cite{alon2013approximate}. In \cref{theorem:compression}, note that such matrices can be represented as $\Matrix{F} = \Matrix{A} \Matrix{A}^\trans$ and $\Norm{\Matrix{A}}{2,\infty}^2 \leq \max_{i \in [n]}|\Matrix{F}(i,i)|$.
\end{remark}

\cref{theorem:kernel_function} lifts the restrictions on $\varepsilon$ and removes the $\log(\varepsilon^{-1})$ factor from the rank bounds for the Gaussian, Cauchy, and exponential kernels (cf. \cref{remark:gaussian}). Meanwhile, there are innocently looking functions like $f(\Point{x}, \Point{y}) = \exp(-\Norm{\Point{x} - \Point{y}}{2}^4)$ that are not positive-definite kernels and therefore cannot be analyzed with \cref{theorem:kernel_function} (but can be with \cref{theorem:1d_function_dist_gen}). Indeed, $f$ violates Schoenberg's criterion \cite{schaback2001characterization}, and thus there exists\footnote{We do not know if there is any value of $m$ for which $f$ is a positive-definite kernel in $\Real^m$.} $m \in \N$ such that it is not a positive-definite kernel in $\Real^m$. For example, this is the case for $m = 1$: according to the numerical evidence \cite{boyd2013quartic}, the function violates Bochner's criterion \cite{schaback2001characterization}. 

%% file: parts/tensors.tex
\section{Function-generated tensors}
\label{sec:tensors}
Many datasets are naturally arranged as tensors with $d \geq 3$ modes. We continue our investigation by analyzing the entrywise low-rank approximation of tensors $\Tensor{F} \in \Real^{n_1 \times \cdots \times n_d}$ generated with a function $f : \Omega_1 \times \cdots \times \Omega_d \to \Real$ as
\begin{equation*}
    \Tensor{F}(i_1, \ldots, i_d) = f(\Point{x}^{(1)}_{i_1}, \ldots, \Point{x}^{(d)}_{i_d}), \quad \{ \Point{x}^{(j)}_{i_j} \}_{i_j = 1}^{n_j} \subset \Omega_j, \quad j \in [d].
\end{equation*}
Given the success of inner-product generating functions in the matrix case, we turn to generating functions that depend only on the ``\emph{higher-order inner product}'' of their variables:
\begin{equation}
\label{eq:hoip}
    \langle \Point{x}^{(1)}, \ldots, \Point{x}^{(d)} \rangle = \sum_{\alpha = 1}^{m} \prod_{j = 1}^{d} \Point{x}^{(j)}(\alpha).
\end{equation}
Thus, we shall focus on generating functions $f(\Point{x}^{(1)}, \ldots, \Point{x}^{(d)}) = h(\langle \Point{x}^{(1)}, \ldots, \Point{x}^{(d)} \rangle)$. They appear in the context of tensor kernels \cite{salzo2018solving} and tensor attention \cite{alman2023capture, sanford2024representational}; see \Cref{subsec:attention}.

Let us describe the tensor that, upon the entrywise application of $h$, becomes $\Tensor{F}$. For each $j \in [d]$, consider a matrix $\Matrix{X}_j \in \Real^{n_j \times m}$ given by $\Matrix{X}_j^\trans = \begin{bmatrix}\Point{x}^{(j)}_1 & \cdots & \Point{x}^{(j)}_{n_j}\end{bmatrix}$. We define a tensor $\Tensor{P} = \llbracket \Matrix{X}_1, \ldots, \Matrix{X}_d \rrbracket \in \Real^{n_1 \times \cdots \times n_d}$ with entries
\begin{equation}
\label{eq:cp}
    \Tensor{P}(i_1, \dots, i_d) = \sum_{\alpha = 1}^{m} \Matrix{X}_1(i_1, \alpha) \dots \Matrix{X}_d(i_d, \alpha) = \langle \Point{x}^{(1)}_{i_1}, \dots, \Point{x}^{(d)}_{i_d} \rangle.
\end{equation}
Such representation of a tensor is known as its CP decomposition of length $m$ ($m$ need not be minimal). See \cite{kolda2009tensor,ballard2025tensor} for an introduction to tensor decompositions.

We need to specify a particular tensor decomposition with which to approximate $\Tensor{F} = h(\Tensor{P}) = h(\llbracket \Matrix{X}_1, \ldots, \Matrix{X}_d \rrbracket)$. All most widely used tensor decompositions reduce to the usual low-rank matrix decomposition when $d=2$, but for $d \geq 3$ they exhibit different properties. We choose the TT decomposition \cite{oseledets2009breaking, oseledets2011tensor}, also known as matrix product states \cite{vidal2003efficient, schollwock2011density}:
\begin{equation*}
    \Tensor{F}(i_1, \dots, i_d) = \sum_{\alpha_1 = 1}^{r_1} \dots \sum_{\alpha_{d-1} = 1}^{r_{d-1}} \Tensor{G}_1(1, i_1, \alpha_1) \Tensor{G}_2(\alpha_1, i_2, \alpha_2) \dots \Tensor{G}_d(\alpha_{d-1}, i_d, 1).
\end{equation*}
The tuple $(r_1, \ldots, r_{d-1}) \in \N^{d-1}$ is called the TT rank of the decomposition. The smallest (componentwise) TT rank among all TT decompositions of $\Tensor{F}$ is called its TT rank $\ttrank{\Tensor{F}}$.

Having fixed the initial (CP) and final (TT) tensor decompositions, we now require a suitable analog of \cref{theorem:compression}. The tensor norm $\Norm{\cdot}{\max}$ is defined just as for matrices.

\begin{theorem}[{\cite[Corollary 1.8]{budzinskiy2025entrywise}}]
\label{theorem:compression_tensor}
Let $\varepsilon \in (0,1)$ and $n_1, \ldots, n_d \in \N$. Consider
\begin{equation}
\label{eq:compression_rank_tensor}
    r = \left\lceil \frac{c_d}{\varepsilon^2} \log(2e \cdot n_1 \dots n_d ) \right\rceil \in \N
\end{equation}
where $c_d > 0$ is an absolute constant that depends\footnote{We do not know the exact law, but the provable estimates of $c_d$ grow at least as $d^d$} only on $d$ and $e$ is Euler's number. For every $m \in \N$ and every $\Tensor{A} \in \Real^{n_1 \times \cdots \times n_d}$ with a CP decomposition $\Tensor{A} = \llbracket \Matrix{A}_1, \ldots, \Matrix{A}_d \rrbracket$ of length $m$, there exists $\Tensor{G} \in \Real^{n_1 \times \cdots \times n_d}$ of $\ttrank{\Tensor{G}} \preccurlyeq r$ such that
\begin{equation*}
    \Norm{\Tensor{A} - \Tensor{G}}{\max} \leq \varepsilon \cdot \prod_{j = 1}^{d} \Norm{\Matrix{A}_j}{2,\infty}.
\end{equation*}
\end{theorem}

\begin{lemma}
\label{lemma:compression_hadamard_tensor}
Let $\varepsilon \in (0,1)$, $n_1, \ldots, n_d \in \N$, and $r \in \N$ be given by \cref{eq:compression_rank_tensor}. For every $t \in \N$, any $m_1, \ldots, m_t \in \N$, and any $\Tensor{A}_1, \ldots, \Tensor{A}_t \in \Real^{n_1 \times \cdots \times n_d}$ with CP decompositions $\Tensor{A}_{s} = \llbracket \Matrix{A}^{(s)}_1, \ldots, \Matrix{A}^{(s)}_d \rrbracket$ of length $m_s$, there exists $\Tensor{G} \in \Real^{n_1 \times \cdots \times n_d}$ of $\ttrank{\Tensor{G}} \preccurlyeq r$~such~that
\begin{equation*}
    \Norm{\Tensor{A}_1 \hap \cdots \hap \Tensor{A}_t - \Tensor{G}}{\max} \leq \varepsilon \prod_{s = 1}^{t} \prod_{j = 1}^{d} \Norm{\Matrix{A}^{(s)}_j}{2,\infty}.
\end{equation*}
\end{lemma}
\begin{proof}
As in the proof of \cref{lemma:compression_hadamard}, observe that $\Tensor{A}_1 \hap \cdots \hap \Tensor{A}_t = \llbracket \tilde{\Matrix{A}}_1, \ldots, \tilde{\Matrix{A}}_d \rrbracket$ with
\begin{equation*}
    \tilde{\Matrix{A}}_j^\trans = [\Matrix{A}^{(1)}_j]^\trans \khrap \cdots \khrap [\Matrix{A}^{(t)}_j]^\trans, \quad 1 \leq j \leq d.
\end{equation*}
Apply \cref{theorem:compression_tensor} and use $\Norm{\tilde{\Matrix{A}}_j}{2,\infty} \leq \prod_{s = 1}^{t} \Norm{\Matrix{A}^{(s)}_j}{2,\infty}$.
\end{proof}

\begin{lemma}
\label{lemma:1d_function_dotp_tensor}
Let $\varepsilon \in (0,1)$, $n_1, \ldots, n_d \in \N$, and $r \in \N$ be given by \cref{eq:compression_rank_tensor}. Let $h : \Omega \to \Real$ satisfy \cref{assumption:1d_smooth} with $0 \in \Omega$. For any $t,m \in \N$ and every $\Tensor{A} \in \Real^{n_1 \times \cdots \times n_d}$ with a CP decomposition $\Tensor{A} = \llbracket \Matrix{A}_1, \ldots, \Matrix{A}_d \rrbracket$ of length $m$ such that $\mathcal{I} = [\min\{\Tensor{A}\}, \max\{\Tensor{A}\}] \subseteq \Omega$, there exists $\Tensor{G} \in \Real^{n_1 \times \cdots \times n_d}$ of $\ttrank{\Tensor{G}} \preccurlyeq 1 + (t-1)r$ such that
\begin{equation*}
    \Norm{h(\Tensor{A}) - \Tensor{G}}{\max} \leq \frac{\Norm{\Tensor{A}}{\max}^t}{t!} \Norm{h^{(t)}}{\Ell_\infty(\Omega)} + \varepsilon \sum_{s = 1}^{t-1} \frac{|h^{(s)}(0)|}{s!} \prod_{j = 1}^{d} \Norm{\Matrix{A}_j}{2,\infty}^s.
\end{equation*}
\end{lemma}
\begin{proof}
    Repeat the proof of \cref{lemma:1d_function_dotp} using \cref{lemma:compression_hadamard_tensor} and note that the TT rank of a sum of two tensors is upper bounded by the sum of their TT ranks \cite{oseledets2011tensor}.
\end{proof}

\begin{corollary}
\label{corollary:1d_function_dotp_tensor}
In the setting of \cref{lemma:1d_function_dotp_tensor}, let $h : \Omega \to \Real$ satisfy \cref{assumption:1d_analytic}. If $\max_{j \in [d]}\{\Norm{\Matrix{A}_j}{2,\infty}\} \leq R$ then, for $t \geq e^2 M R^d$, 
\begin{equation*}
    \Norm{h(\Tensor{A}) - \Tensor{G}}{\max} \leq C \Big[ e^{-t} + \varepsilon (e^{M R^d} - 1) \Big].
\end{equation*}
\end{corollary}
\begin{proof}
Use \cref{lemma:1d_function_dotp_tensor} and proceed as in \cref{corollary:1d_function_dotp}, noting that $\Norm{\Tensor{A}}{\max} \leq R^d$.
\end{proof}

\begin{theorem}
\label{theorem:1d_function_tensor}
Let $\varepsilon \in (0,1)$, $n_1, \ldots, n_d \in \N$, and $r \in \N$ be given by \cref{eq:compression_rank_tensor}. Let $h : \Omega \to \Real$ satisfy \cref{assumption:1d_analytic} with $[-R^d, R^d] \subseteq \Omega$ and $R > 0$. For every $m \in \N$, every $\Tensor{F}~\in~\Real^{n_1 \times \cdots \times n_d}$ generated with $f(\Point{x}^{(1)}, \ldots, \Point{x}^{(d)}) = h(\langle \Point{x}^{(1)}, \ldots, \Point{x}^{(d)} \rangle)$ based on points from $\overline{\ball}_{m,R}$, and every integer $t \geq e^2 M R^d$, there exists $\Tensor{G} \in \Real^{n_1 \times \cdots \times n_d}$ of $\ttrank{\Tensor{G}} \preccurlyeq 1 + (t-1)r$ such that
\begin{equation*}
    \Norm{\Tensor{F} - \Tensor{G}}{\max} \leq C \Big[ e^{-t} + \varepsilon (e^{M R^d} - 1) \Big].
\end{equation*}
\end{theorem}
\begin{proof}
Apply \cref{corollary:1d_function_dotp_tensor} to the tensor $\Tensor{P} = \llbracket \Matrix{X}_1, \ldots, \Matrix{X}_d \rrbracket$ introduced in \cref{eq:cp}.
\end{proof}

Similar to matrices, we can choose $t = \lceil \log(\varepsilon^{-1}) \rceil$ to guarantee the existence of a tensor $\Tensor{G}$ of $\ttrank{\Tensor{G}} = \mathcal{O}(\log(n_1 \ldots n_d) \varepsilon^{-2} \log(\varepsilon^{-1}))$ that achieves the error $\Norm{\Tensor{F} - \Tensor{G}}{\max} = \mathcal{O}(\varepsilon)$. The TT rank bound is independent of the dimension $m$ and deteriorates with the growth of the number of modes $d$ because of the constant $c_d$ in \cref{theorem:compression_tensor}.

%% file: parts/numerical.tex
\section{Numerical experiments}
\label{sec:numerical}
\subsection{Algorithm}
While the truncated singular value decomposition (SVD) of a matrix provides its optimal low-rank approximation in the Frobenius norm, optimal low-rank approximation in the maximum norm is substantially more difficult---to the extent that the problem is NP-hard even in the simplest rank-one case \cite{gillis2019low, zamarashkin2022best, morozov2023optimal}. It might still be possible to find ``good'' entrywise approximations, and we focus on a heuristic method based on \emph{alternating projections} as developed in \cite[\S 7.5]{budzinskiy2023quasioptimal} and further used in \cite{budzinskiy2024distance, budzinskiy2025entrywise}. This approach applies to both matrices and tensors and, to the best of our knowledge, is the only currently available specialized algorithm of entrywise low-rank approximation for tensors. Find an overview of the known matrix algorithms in \cite{budzinskiy2024distance}.

The method of alternating projections aims to find a point in the intersection of two sets by computing successive projections in the Frobenius norm. Consider the sets of low-rank tensors and tensors that are close to $\Tensor{F}$ in the maximum norm:
\begin{equation*}
    \mathbb{M}_{\bm{r}} = \set{\Tensor{T}}{\ttrank{\Tensor{T}} = \bm{r}}, \quad \mathbb{B}_\varepsilon(\Matrix{F}) = \set{\Tensor{E}}{\Norm{\Tensor{F} - \Tensor{E}}{\max} \leq \varepsilon}.
\end{equation*}
The projection of $\Tensor{T}$ onto $\mathbb{B}_\varepsilon(\Tensor{F})$ amounts to clipping the entries of $\Tensor{T} - \Tensor{F}$ whose absolute value exceeds $\varepsilon$. The projection onto $\mathbb{M}_{\bm{r}}$ is given by the truncated SVD for matrices. It cannot be computed for tensors, but the TT-SVD algorithm \cite{oseledets2009breaking, oseledets2011tensor} acts as a quasi-optimal low-rank projection, and the iterations of quasi-optimal alternating projections with TT-SVD instead of the optimal projection still converge \cite{sultonov2023low,budzinskiy2023quasioptimal} and can be accelerated with randomized low-rank projections \cite{matveev2023sketching, sultonov2023low}. In \cite{budzinskiy2023quasioptimal, budzinskiy2025entrywise, budzinskiy2024distance}, (quasi-optimal) alternating projections were combined with a binary search over $\varepsilon$ to obtain an upper bound on the optimal low-rank approximation error in the maximum norm for a matrix or tensor.

\subsection{Test functions and experimental settings}

We consider three functions,\footnote{Note the discussion about $f_2$ at the end of \Cref{sec:kernels}.}
\begin{equation*}
    f_1(\Point{x}, \Point{y}) = \exp(-\Norm{\Point{x} - \Point{y}}{2}), \quad f_2(\Point{x}, \Point{y}) = \exp(-\Norm{\Point{x} - \Point{y}}{2}^4), \quad f_3(\Point{x}, \Point{y}, \Point{z}) = \sinh(\langle \Point{x}, \Point{y}, \Point{z} \rangle),
\end{equation*}
and generate $n \times n$ matrices and $n \times n \times n$ tensors by sampling the points $\{ \Point{x}_i \}$, $\{ \Point{y}_i \}$, $\{ \Point{z}_i \}$ uniformly at random from $\ball_m$. With matrices, we use two sampling schemes: independent sampling when $\{ \Point{x}_i \}$ and $\{ \Point{y}_i \}$ are sampled independently, and symmetric sampling when we sample $\{ \Point{x}_i \}$ and set $\Point{y}_i = \Point{x}_i$. With tensors, we use only independent sampling. 

In every numerical experiment, we set the dimension $m$ and the size $n$; choose the sampling scheme and the test function; generate a matrix or tensor from random samples; set the approximation rank $r$ (TT rank ($r$, $r$) for tensors) and run the method of alternating projections with binary search starting from a low-rank initial approximation with random Gaussian factors; normalize the resulting error estimate by the maximum norm of the original matrix or tensor to obtain the relative error estimate $\varepsilon$. We repeat every experiment 5 times and report the median value of $\varepsilon$. For matrices, we additionally compare the results of alternating projections with the entrywise errors achieved by the truncated SVD (again, all random experiments are repeated 5 times); this estimate was used in \cite{udell2019big}.

Our implementation\footnote{The code is available at \url{https://github.com/sbudzinskiy/low-rank-big-data}.} of the method of alternating projections with binary search has a number of hyperparameters that affect its performance and running time. The choice of their values was dictated in large by the latter, so, perhaps, the error estimates can be improved. We restrict the experiments to third-order tensors for the same reasons. Similarly, instead of fixing the desired relative approximation error $\varepsilon$ and studying the $\varepsilon$-rank \cref{eq:epsrank}, we fix the approximation rank $r$ and bound the optimal entrywise error of rank-$r$ approximation, since the second problem is less computationally challenging.

Let us repeat the comment made in \Cref{subsec:datasets}: the values of $n$ and $r$, for which we have enough computational resources to perform experiments, are too low to fall into the regime described by our theory. Therefore, we suggest to view our numerical results as complementary to the theoretical bounds.

\subsection{Varying dimension of the variables}

\begin{figure}[t]
\centering
	\includegraphics[width=0.94\textwidth]{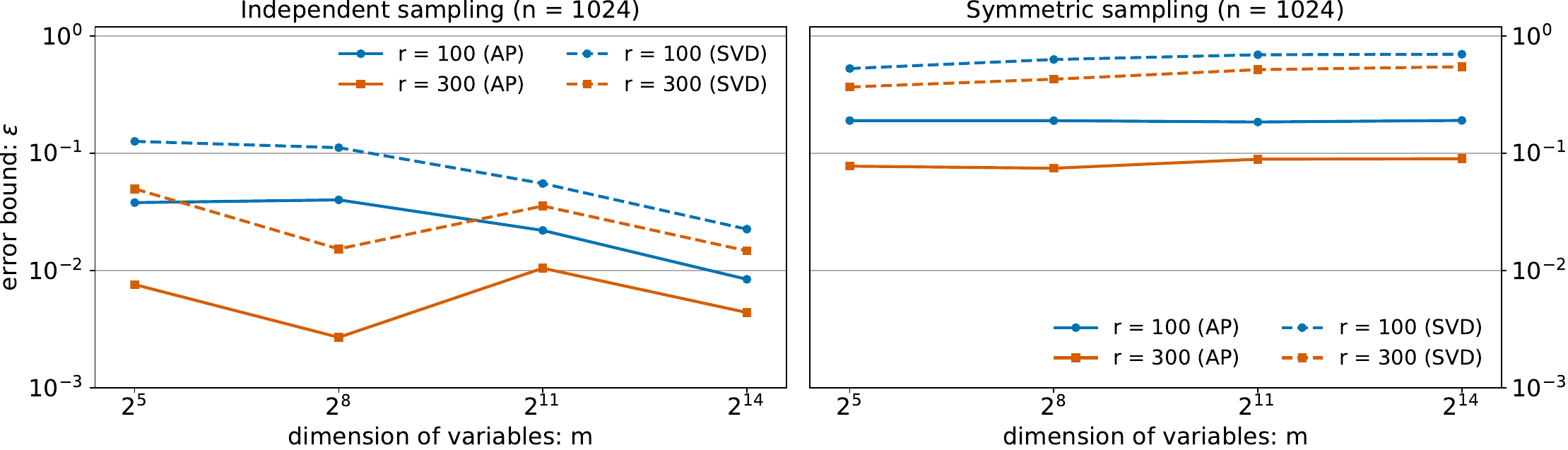}
	\includegraphics[width=0.94\textwidth]{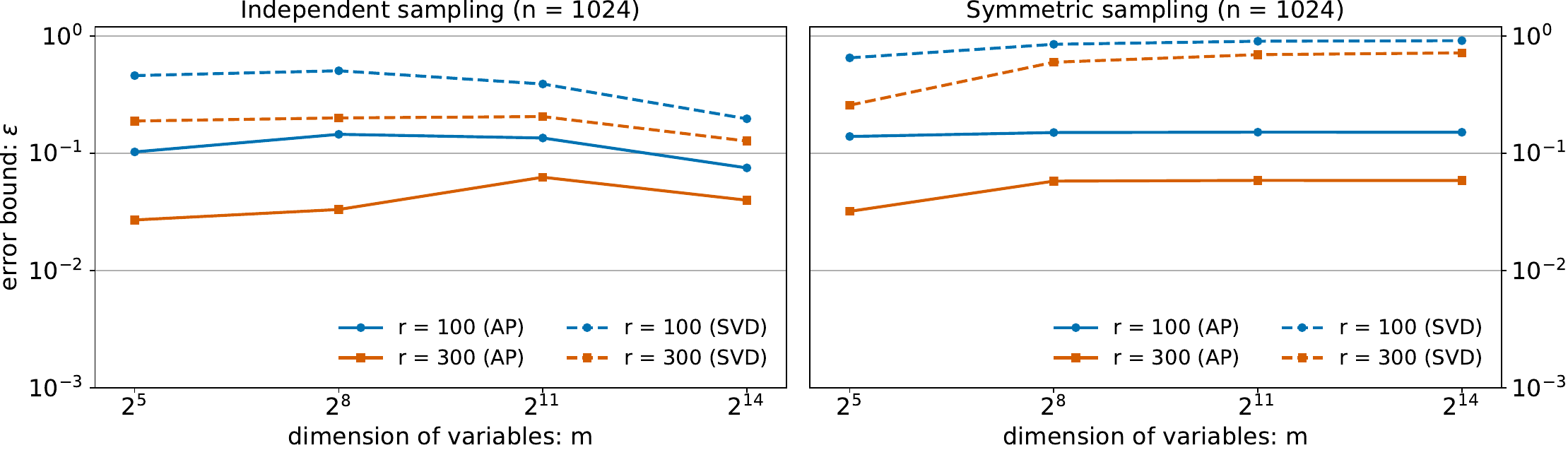}
	\includegraphics[width=0.47\textwidth]{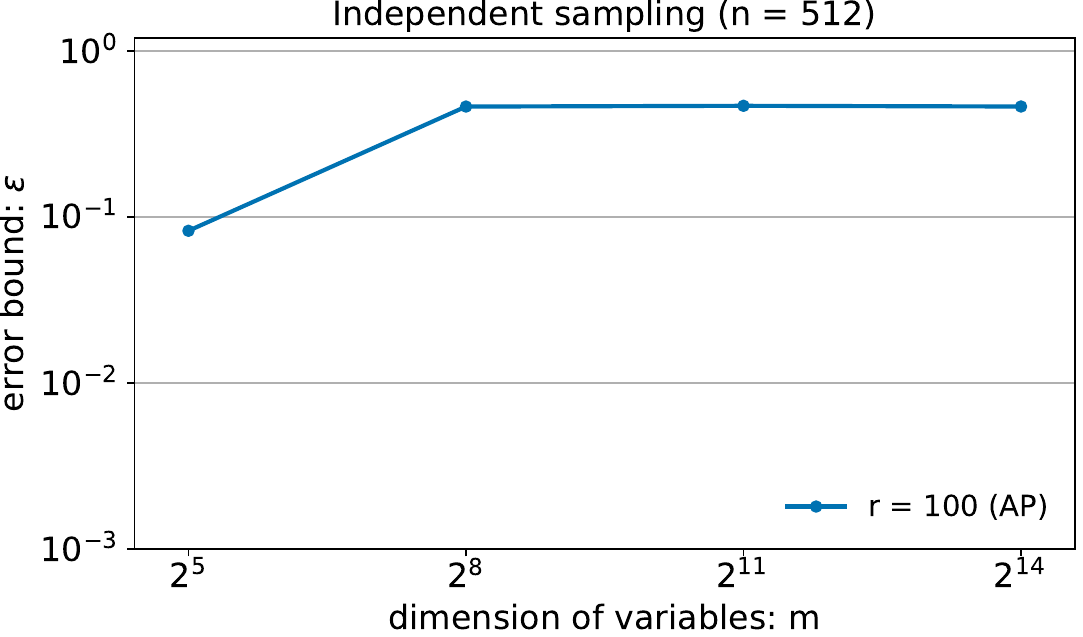}
\caption{Relative errors of low-rank approximation in the maximum norm for function-generated matrices and tensors with varying dimension $m$: (top) $f_1(\Point{x}, \Point{y}) = \exp(-\Norm{\Point{x} - \Point{y}}{2})$, (middle) $f_2(\Point{x}, \Point{y}) = \exp(-\Norm{\Point{x} - \Point{y}}{2}^4)$, (bottom) $f_3(\Point{x}, \Point{y}, \Point{z}) = \sinh(\langle \Point{x}, \Point{y}, \Point{z} \rangle)$.}
\label{fig:varm}
\end{figure}

The first series of experiments illustrates that, for specific classes of generating functions, the dimension $m$ of the variables does not affect the rank bound of entrywise approximation of the corresponding function-generated matrices and tensors. As \cref{fig:varm} shows, the errors remain approximately constant as $m$ is increased. When $m$ is large, matrices generated with $f_1$ and $f_2$ according to the random symmetric sampling scheme behave like identity matrices, which explains why their approximation errors are similar to those reported for identity matrices in \cite{budzinskiy2024distance}. Note that our algorithm achieves approximation errors that are 5-10 times smaller than those of the truncated SVD at the cost of higher computational complexity (recall that the truncated SVD is optimal only for unitarily invariant norms and that the maximum norm is not among them).

\subsection{Varying approximation rank}

\begin{figure}[t]
\centering
	\includegraphics[width=0.97\textwidth]{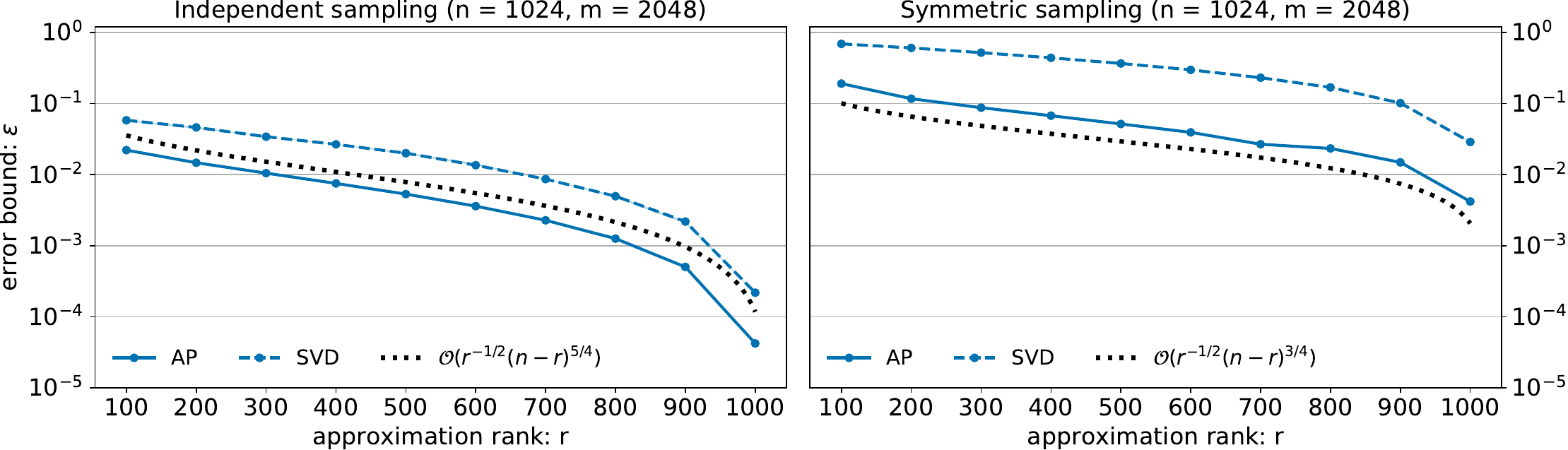}
	\includegraphics[width=0.97\textwidth]{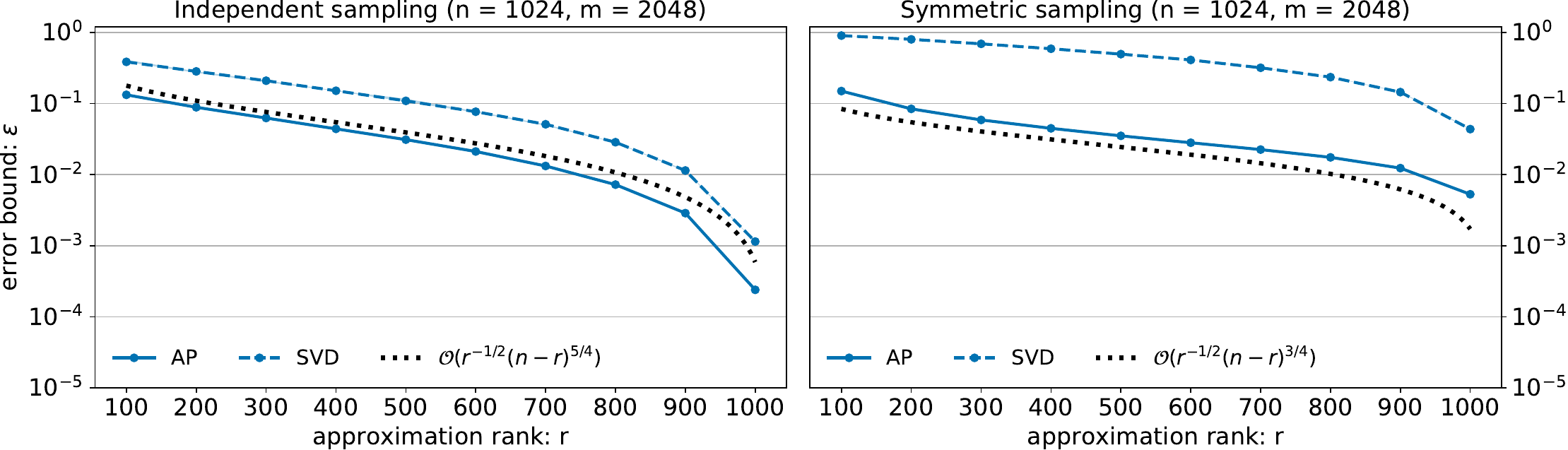}
	\includegraphics[width=0.485\textwidth]{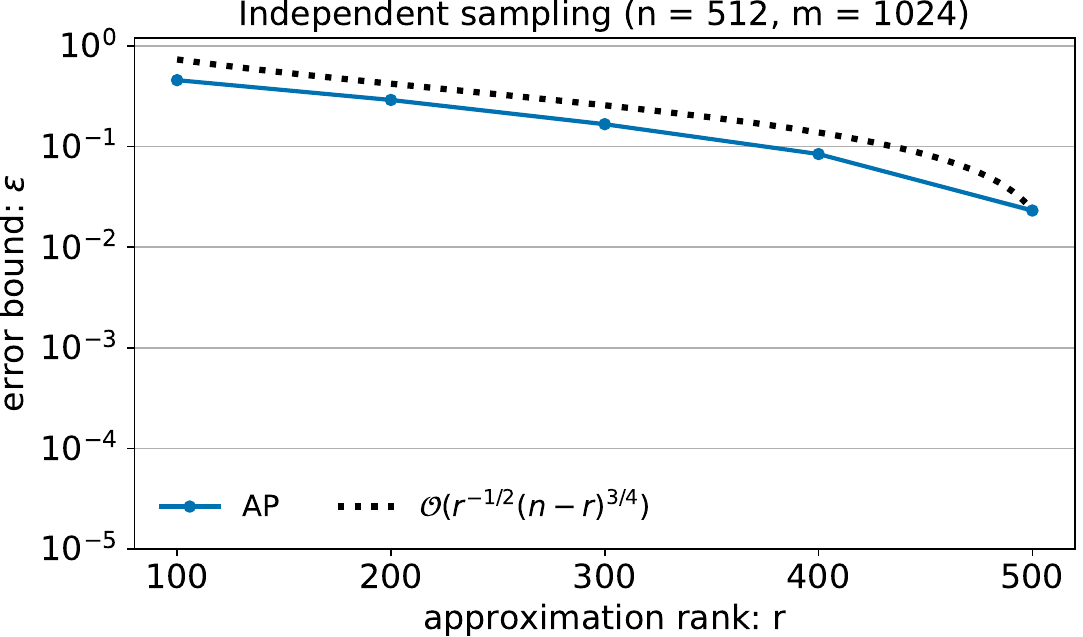}
\caption{Relative errors of low-rank approximation in the maximum norm for function-generated matrices and tensors with varying approximation rank $r$: (top) $f_1(\Point{x}, \Point{y}) = \exp(-\Norm{\Point{x} - \Point{y}}{2})$, (middle) $f_2(\Point{x}, \Point{y}) = \exp(-\Norm{\Point{x} - \Point{y}}{2}^4)$, (bottom) $f_3(\Point{x}, \Point{y}, \Point{z}) = \sinh(\langle \Point{x}, \Point{y}, \Point{z} \rangle)$.}
\label{fig:varr}
\end{figure}

Next, we study the decay of the entrywise errors as the approximation rank $r$ grows. In \cref{fig:varr}, we see that the errors decrease as $\mathcal{O}(r^{-1/2}(n-r)^\alpha)$, where the hidden constant depends on $n$ and $\alpha > 0$ depends on the sampling regime and the order of the tensor. For the values of $r$ that are not too close to the size $n$, the numerically estimated errors decay as $r^{-1/2}$ in accordance with our theoretical guarantee \cref{eq:rank_bound}.

\subsection{Varying size of the matrix or tensor}

\begin{figure}[t]
\centering
	\includegraphics[width=0.97\textwidth]{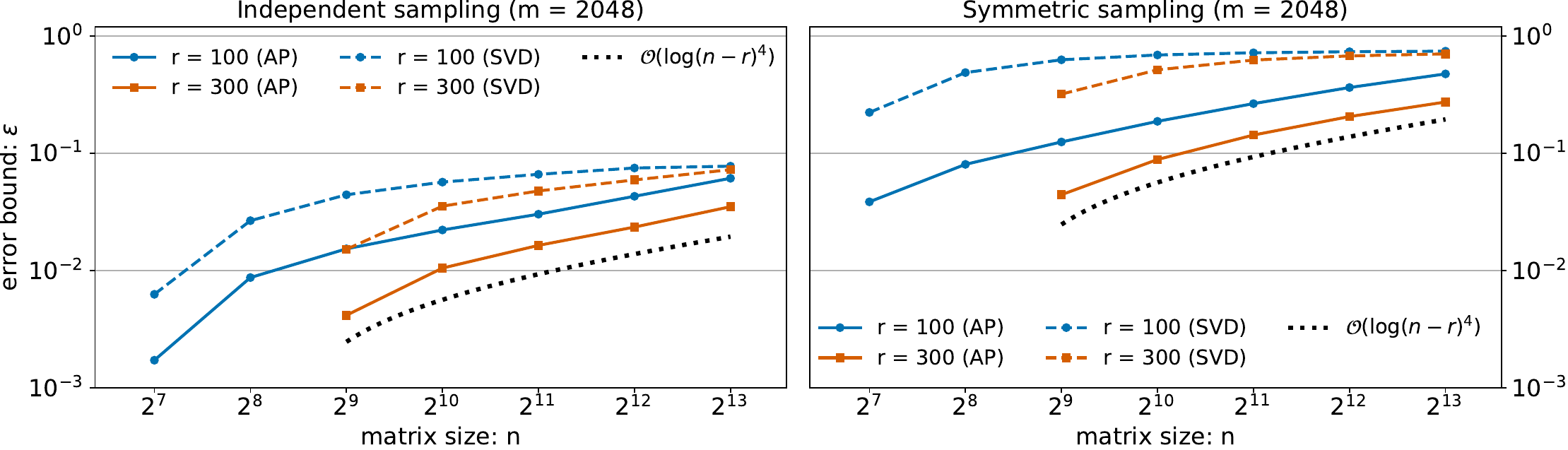}
	\includegraphics[width=0.97\textwidth]{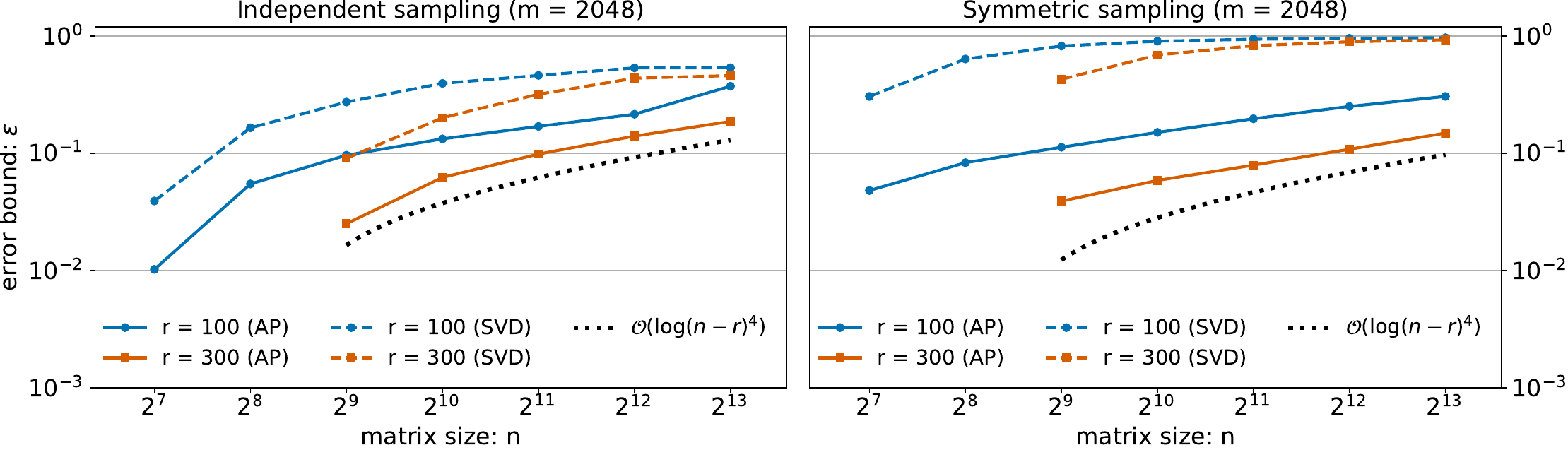}
	\includegraphics[width=0.485\textwidth]{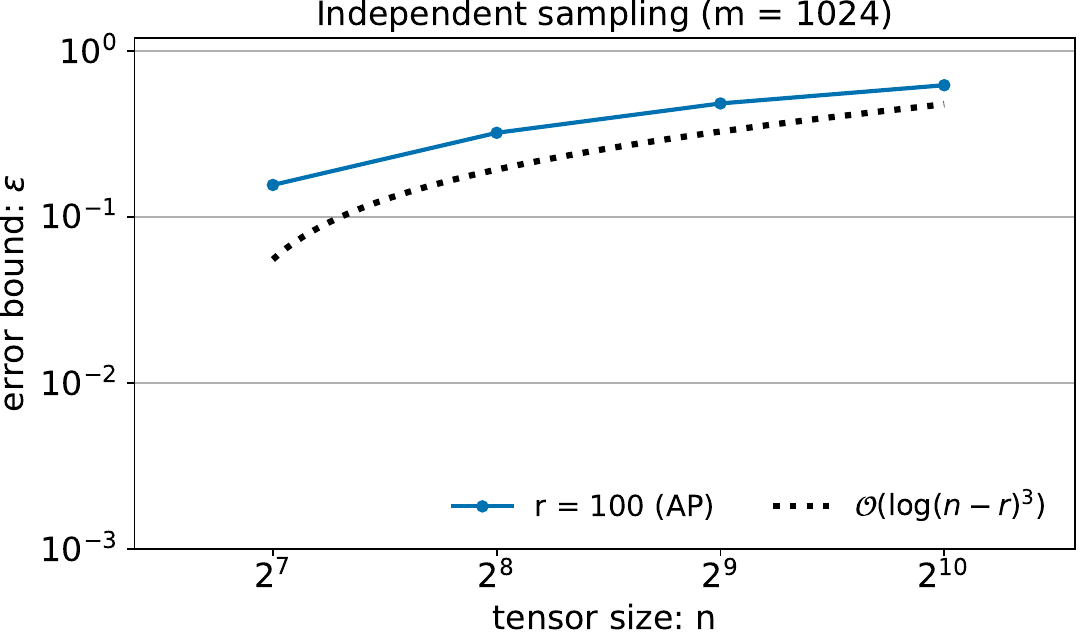}
\caption{Relative errors of low-rank approximation in the maximum norm for function-generated matrices and tensors with varying size $n$: (top) $f_1(\Point{x}, \Point{y}) = \exp(-\Norm{\Point{x} - \Point{y}}{2})$, (middle) $f_2(\Point{x}, \Point{y}) = \exp(-\Norm{\Point{x} - \Point{y}}{2}^4)$, (bottom) $f_3(\Point{x}, \Point{y}, \Point{z}) = \sinh(\langle \Point{x}, \Point{y}, \Point{z} \rangle)$.}
\label{fig:varn}
\end{figure}

Finally, we investigate how the approximation errors depend on the size $n$ of the function-generated matrix or tensor. Our theoretical rank bound \cref{eq:rank_bound} suggests that, for fixed approximation rank $r$, it is possible to achieve an entrywise error of $\mathcal{O}(\sqrt{\log(n) / r})$. On the other hand, there is an ultimate upper bound $\Norm{\Matrix{F}}{\max}$ of the error that is attained at the rank-zero approximation and can be approached with a vanishing sequence of rank-$r$ matrices. Therefore, the error bound $\mathcal{O}(\sqrt{\log(n) / r})$ necessarily becomes meaningless when $r$ is fixed and $n \to \infty$. \cref{fig:varn} shows the growth rate $\mathcal{O}(\mathrm{polylog}(n-r))$ of the approximation error in the intermediate range of values of $n$. We expect that the rate $\mathcal{O}(\sqrt{\log(n) / r})$ could be registered for bigger values of $n$ and $r$ (see \Cref{subsec:datasets}).

%% file: parts/conclusion.tex
\section{Conclusion}
So are ``big data'' approximately low-rank? The question is too vague to have a concise, definite answer. In this article, we presented three classes of function-generated matrices and a class of function-generated tensors whose $\varepsilon$-ranks tend to grow slowly with the number of samples $n$ (once $n$ is big enough) when the underlying latent dimension $m$ is large. It is unlikely that these specific types of function-generated data represent the whole of ``big data'', and we encourage further investigations in this area: seeking other classes of approximately low-rank matrices and tensors and tying them with the real-world datasets.

%% file: parts/acknowledgements.tex
\section*{Acknowledgements}
I thank Dmitry Kharitonov for double-checking some of the presented arguments and Vladimir Kazeev for his valuable comments and feedback. I am grateful to Madeleine Udell for the discussion of \cite{udell2019big} and to the anonymous referees, whose comments have helped me to improve the presentation and tighten the rank bounds in \Cref{sec:dist}.

%% file: parts/appendix.tex
\section{Multi-index notation}
\label{appendix:multindex}
Consider a smooth function $f : \Real^m \to \Real$, a point $\Point{x} \in \Real^m$, and a multi-index $\bm{\gamma} = (\gamma_1, \ldots, \gamma_m) \in \N^m_0$. The following notation is standard in multi-variate calculus and partial differential equations:
\begin{equation*}
    \bm{\gamma}! = \gamma_1! \cdots \gamma_m!, \quad \Point{x}^{\bm{\gamma}} = \Point{x}(1)^{\gamma_1} \cdots \Point{x}(m)^{\gamma_m}, \quad \partial^{\bm{\gamma}} f = \partial_{1}^{\gamma_1} \cdots \partial_{m}^{\gamma_m} f.
\end{equation*}
With its help, we can write down the Taylor series of an analytic function in a concise way:
\begin{equation*}
    f(\Point{x}) = \sum\nolimits_{\bm{\gamma} \in \N^m_0} \partial^{\bm{\gamma}} f(\Point{0}) \frac{\Point{x}^{\bm{\gamma}}}{\bm{\gamma}!}.
\end{equation*}
Using $|\bm{\gamma}| = \gamma_1 + \cdots + \gamma_m$, we can also write the truncated Taylor series:
\begin{equation*}
    f(\Point{x}) \approx \sum\nolimits_{|\bm{\gamma}| < \rho} \partial^{\bm{\gamma}} f(\Point{0}) \frac{\Point{x}^{\bm{\gamma}}}{\bm{\gamma}!}.
\end{equation*}

\section{Proof of Theorem~\ref{theorem:ut_tighter}}
\label{appendix:proof_ut}
\begin{proof}
For each $\rho \in \N$ and $\Point{x} \in \ball_m$, consider the Taylor series of $\Point{y} \mapsto f(\Point{x},\Point{y})$ at $\Point{y} = \Point{0}$,
\begin{equation*}
    f(\Point{x}, \Point{y}) = \sum_{|\bm{\gamma}| < \rho} \partial_{\Point{y}}^{\bm{\gamma}} f(\Point{x}, \Point{0}) \frac{\Point{y}^{\bm{\gamma}}}{\bm{\gamma}!} + E_{\rho}(\Point{x}, \Point{y}) = f_{\rho}(\Point{x}, \Point{y}) + E_{\rho}(\Point{x}, \Point{y}), \quad \Point{y} \in \ball_m,
\end{equation*}
with the remainder $E_{\rho}(\Point{x}, \Point{y})$ given in the Lagrange form as
\begin{equation*}
    E_{\rho}(\Point{x}, \Point{y}) = \sum_{|\bm{\gamma}| = \rho} \partial_{\Point{y}}^{\bm{\gamma}} f(\Point{x}, \alpha \Point{y}) \frac{\Point{y}^{\bm{\gamma}}}{\bm{\gamma}!}, \quad \alpha = \alpha(\Point{y}) \in (0,1).
\end{equation*}
Since the point $\alpha \Point{y}$ lies in $\ball_m$,
\begin{equation*}
    |E_{\rho}(\Point{x}, \Point{y})| \leq \sum_{|\bm{\gamma}| = \rho} \Norm{\partial_{\Point{y}}^{\bm{\gamma}} f(\Point{x}, \cdot)}{\Ell_\infty(\ball_m)} \frac{|\Point{y}^{\bm{\gamma}}|}{\bm{\gamma}!} \leq C M^{\rho} \Norm{f}{\Ell_\infty(\ball_m^2)} \sum_{|\bm{\gamma}| = \rho} \frac{|\Point{y}^{\bm{\gamma}}|}{\bm{\gamma}!}.
\end{equation*}
We have $\max_{i \in [m]} |\Point{y}(i)| \leq 1$ so that
\begin{equation*}
    \Norm{E_{\rho}(\Point{x},\cdot)}{\Ell_\infty(\ball_m)} \leq C M^{\rho} \Norm{f}{\Ell_\infty(\ball_m^2)} \sum_{|\bm{\gamma}| = \rho} \frac{1}{\bm{\gamma}!} = C M^{\rho} \frac{m^{\rho}}{\rho!} \Norm{f}{\Ell_\infty(\ball_m^2)},
\end{equation*}
where we used the multinomial theorem for $m^\rho = (1 + \cdots + 1)^\rho$. Fix any $\delta > 0$. For sufficiently large $\rho = \rho(\delta)$, the truncated Taylor series $f_\rho(\Point{x}, \Point{y})$ satisfies
\begin{equation*}
    \Norm{f - f_\rho}{\Ell_\infty(\ball_m^2)} \leq \delta \Norm{f}{\Ell_\infty(\ball_m^2)}.
\end{equation*}
Denote by $k_{m,\rho} \in \N$ the number of terms in $f_\rho$, i.e., the cardinality of $\set{\bm{\gamma} \in \N_0^m}{|\bm{\gamma}| < \rho}$. Let $\{ \Point{x}_i \}_{i = 1}^{n_1} \subset \ball_m$ and $\{ \Point{y}_j \}_{j = 1}^{n_2} \subset \ball_m$ be the points used to generate the matrix $\Matrix{F} \in \Real^{n_1 \times n_2}$. We can sample the function $f_\rho$ at these points to construct, as in \cref{example:1d}, matrices $\Matrix{X} \in \Real^{n_1 \times k_{m,\rho}}$, $\Matrix{Y} \in \Real^{n_2 \times k_{m,\rho}}$ and diagonal $\Matrix{D} \in \Real^{k_{m,\rho} \times k_{m,\rho}}$ such that
\begin{equation*}
    \Norm{\Matrix{F} - \Matrix{X} \Matrix{D} \Matrix{Y}^\trans}{\max} \leq \delta \Norm{f}{\Ell_\infty(\ball_m^2)}.
\end{equation*} 
Let $\Matrix{A} = \Matrix{X} \Matrix{D}^{\frac{1}{2}}$ and $\Matrix{B} = \Matrix{Y} \Matrix{D}^{\frac{1}{2}}$. Now, let us apply \cref{theorem:compression} to the matrix $\Matrix{A} \Matrix{B}^\trans = \Matrix{X} \Matrix{D} \Matrix{Y}^\trans$. The row-norms of $\Matrix{A}$ can be estimated as
\begin{align*}
    \Norm{\Matrix{A}(i,:)}{2}^2 &\leq \sum_{|\bm{\gamma}| = 0}^{\infty} \frac{1}{\bm{\gamma}!} |\partial_{\Point{y}}^{\bm{\gamma}} f(\Point{x}_i, \Point{0})|^2 \leq C^2 \Norm{f}{\Ell_\infty(\ball_m^2)}^2 \sum_{|\bm{\gamma}| = 0}^{\infty} \frac{1}{\bm{\gamma}!} M^{2|\bm{\gamma}|} \\
    &= C^2 \Norm{f}{\Ell_\infty(\ball_m^2)}^2 \bigg( \sum_{\gamma = 0}^{\infty} \frac{1}{\gamma!} M^{2\gamma} \bigg)^{m} = C^2 e^{m M^2} \Norm{f}{\Ell_\infty(\ball_m^2)}^2.
\end{align*}
Similarly, the row-norms of $\Matrix{B}$ are bounded by
\begin{equation*}
    \Norm{\Matrix{B}(j,:)}{2}^2 \leq \sum_{|\bm{\gamma}| = 0}^{\infty} \frac{1}{\bm{\gamma}!} |\Point{y}_j^{\bm{\gamma}}|^2 \leq \sum_{|\bm{\gamma}| = 0}^{\infty} \frac{1}{\bm{\gamma}!} = e^{m}.
\end{equation*}
Let $\Tilde{\varepsilon} \in (0,1)$. By \cref{theorem:compression}, there exists $\Matrix{G} \in \Real^{n_1 \times n_2}$ of rank at most $r = \lceil 9 \log(3 n_1 n_2) / \Tilde{\varepsilon}^2 \rceil$ such that
\begin{equation*}
    \Norm{\Matrix{A} \Matrix{B}^\trans - \Matrix{G}}{\max} \leq \tilde{\varepsilon} \cdot C e^{\frac{m}{2} (1 + M^2)} \Norm{f}{\Ell_\infty(\ball_m^2)},
\end{equation*}
and, as a consequence,
\begin{equation*}
    \Norm{\Matrix{F} - \Matrix{G}}{\max} \leq (\delta + \tilde{\varepsilon} \cdot C e^{\frac{m}{2} (1 + M^2)}) \Norm{f}{\Ell_\infty(\ball_m^2)}.
\end{equation*}
Note that $\Matrix{G}$ depends on $\delta$, take a vanishing sequence $\{ \delta_l \}$, and repeat the limiting argument from the proof of \cref{theorem:kernel_function}.The result follows once we choose $\tilde{\varepsilon} = \varepsilon C^{-1} e^{-\frac{m}{2} (1 + M^2)}$.
\end{proof}

\section{Comparison of rank bounds}
\label{appendix:worse_bound}

\subsection{Theorem~\ref{theorem:ut} and Equation~\ref{eq:analytical_bound}}

Let us find a sufficient condition for the rank bound of \cref{theorem:ut} to be looser than the rank bound \cref{eq:analytical_bound} for all $n$:
\begin{equation*}
    1 + 8 \log(2n+1) (1 + \tfrac{2}{\varepsilon} (C_u + C_v + 1))^2 > e^m \max\{ C/\varepsilon, (1 + e^2 M)^m \}.
\end{equation*}
Since $C_u > C^2 m^m$ and $C_v > 0$, this would follow from a stronger inequality
\begin{gather*}
    \tfrac{32 C^4}{\varepsilon^2} \log(2n+1) m^{2m} > e^m \max\{ C/\varepsilon, (1 + e^2 M)^m \}, \\
    \log(2n+1) > \tfrac{1}{32} \max\left\{ \tfrac{\varepsilon}{C^3} \tfrac{e^m}{m^{2m}}, \tfrac{\varepsilon^2}{C^4} \tfrac{e^m (1 + e^2 M)^m}{m^{2m}} \right\}.
\end{gather*}
As $\varepsilon \in (0,1)$ and $C \geq 1$, this would follow from a yet stronger inequality
\begin{equation*}
    \log(2n+1) > \left[ \tfrac{e(1 + e^2 M)}{m^2} \right]^m.
\end{equation*}
If the right-hand side is smaller than one, the inequality holds for all $n \in \N$. This gives a sufficient condition $M < (m^2 - e)/e^3$ and requires $m \geq 2$ to have $M > 0$.

\subsection{Theorem~\ref{theorem:ut_tighter} and Equation~\ref{eq:analytical_bound}}

Let us find a sufficient condition for the rank bound of \cref{theorem:ut_tighter} corresponding to the error $\varepsilon \Norm{f}{\Ell_\infty}$ to be looser than \cref{eq:analytical_bound} for all $n$:
\begin{gather*}
    \tfrac{9C^2}{\varepsilon^2} \log(3n^2) e^{m(1+M^2)} > e^m \max\{ C / \varepsilon, (1 + e^2 M)^m \}, \\
    \log(3n^2) > \tfrac{1}{9} \max\left\{ \tfrac{\varepsilon}{C} e^{-m M^2}, \tfrac{\varepsilon^2}{C^2} \left( \tfrac{1 + e^2 M}{e^{M^2}} \right)^m \right\}.
\end{gather*}
As $\varepsilon \in (0,1)$ and $C \geq 1$, this would follow from a stronger inequality
\begin{equation*}
    \log(3n^2) > \max\left\{ e^{-m M^2}, \left( \tfrac{1 + e^2 M}{e^{M^2}} \right)^m \right\}.
\end{equation*}
If the right-hand side is smaller than one, the inequality holds for all $n \in \N$. This gives a sufficient condition $e^{M^2} > 1 + e^2 M$, which is true for all $M > 1.597$.